\documentclass[10pt]{amsart}
 \usepackage[foot]{amsaddr}
\usepackage{amssymb}
\usepackage{enumerate,multicol}
\usepackage{mathtools}
\usepackage{varioref}
\usepackage{pifont}
 \usepackage{tikz}
 \usetikzlibrary{calc}
 \usetikzlibrary{arrows}
\usepackage[cmtip,all]{xy}
\usepackage{geometry}
\usepackage{colonequals}

\newtheorem{theorem}{Theorem}[section]
\newtheorem{proposition}[theorem]{Proposition}
\newtheorem{corollary}[theorem]{Corollary}
\newtheorem{lemma}[theorem]{Lemma}
\theoremstyle{definition}
\newtheorem{definition}[theorem]{Definition}
\theoremstyle{remark}
\newtheorem{remark}{Remark}[theorem]
\newtheorem{example}{Example}[theorem]

\newcommand\NN{\mathbb{N}}
\newcommand\ZZ{\mathbb{Z}}
\newcommand\B{\mathcal{B}}
\newcommand\A{\mathcal{A}}
\newcommand\LL{\mathcal{L}}

\DeclareMathOperator\id{id}
\DeclareMathOperator\Ext{Ext}
\DeclareMathOperator\Hom{Hom}
\DeclareMathOperator\im{Im}

\newcommand{\llrr}[1]{
  \langle #1 \rangle}
\title{Projective resolutions of associative algebras and ambiguities}
\author{Sergio Chouhy and Andrea Solotar}
\thanks{This work has been supported by the projects UBACYT X475, PIP-CONICET 2012-2014 11220110100870, PICT 2011-1510  and MathAmSud-GR2HOPF.
The first author is a CONICET fellow.
The second author is a research member of CONICET}
\date{}

\begin{document}

\begin{abstract}
   {
   The aim of this article is to give a method to construct bimodule resolutions of associative algebras, generalizing Bardzell's well-known resolution of monomial algebras.
   We stress that this method leads to concrete computations, providing thus a useful tool for computing invariants associated to the considered algebras. We illustrate how 
   to use it by giving several examples in the last section of the article. In particular we give necessary and sufficient conditions for noetherian down-up algebras to be 3-Calabi-Yau.
      }
\end{abstract}

\maketitle

\medskip




\section{Introduction}
\label{s:Intro}
The invariants attached  to associative algebras and in particular 
to  finite dimensional algebras, have been widely 
studied during the last decades. Among others, Hochschild homology
and cohomology of diverse families of algebras have been computed.

The first problem one faces when computing Hochschild (co)homology is
to find a convenient projective resolution of the algebra as a bimodule over itself.
Of course, the bar resolution is always available but it is almost impossible
to perform computations using it.

M. Bardzell provided in \cite{Ba} a bimodule resolution for monomial algebras, that is, algebras $A=kQ/I$ with $k$ a field, $Q$ a 
finite quiver and $I$ a two-sided ideal which can be generated by monomial 
relations; in this situation, 
the set of classes in $A$ of paths in $Q$ which are not zero is a basis of 
$A$.
Moreover, this resolution is minimal. A simple proof of the exactness
of Bardzell's complex has been given by E. Sk\"oldberg in \cite{Sk},
where he provided a contracting homotopy.
Of course, having such a resolution does not solve the whole problem, it is just a 
starting point. 

The non monomial case is more difficult, since it involves rewriting the paths in
terms of a basis of $A$.
Different kinds of resolutions for diverse families of 
algebras have been provided in the literature.
For augmented $k$-algebras, Anick constructed in \cite{An} a
projective resolution of the ground field $k$. The projective modules
in this resolution are constructed in terms of ambiguities (or $n$-chains),
 and the differentials are not given explicitly. In practice, 
it is hard to make this construction explicit enough in order to compute cohomology.
For quotients of path algebras over a quiver $Q$ with a finite number of
 vertices, Anick and Green exhibited in \cite{AG} a resolution 
for the simple module associated to each vertex, generalizing the 
result of \cite{An}, which deals with the case where the quiver $Q$ has only one vertex.
Also, Y. Kobayashi in \cite{Kob} proposes a method to 
construct a resolution which seems not to be extremely useful.

One may think that the case of binomial algebras is easier than others, but 
in fact it is not quite true since it is necesssary to keep track of 
all reductions performed when writing an element in terms of a chosen 
basis of the algebra as a vector space.

In this article we construct in an inductive way, given an algebra $A$,  a projective
bimodule resolution of $A$, 
which is a kind of deformation of Bardzell's resolution of a monomial algebra associated to $A$. 
For this, we use ideas coming from  Bergman's 
Diamond Lemma and from the theory of Gr\"obner bases. The resolution we give 
is not always minimal, but we prove minimality for various 
families of algebras.

In the context of quotients of path algebras corresponding to a quiver with a finite number of vertices, our method consists in constructing
a resolution whose projective bimodules come from ambiguities present in the rewriting system.
Of course there are many different ways of choosing a basis, so we must state conditions that assure that the rewriting process ends 
and that it is efficient.

One of the advantages of doing this is that, once a bimodule resolution is obtained, it is easy to construct starting from it a resolution of any module on one side 
and, in particular, to recover those constructed in \cite{An} and 
\cite{AG} for the case of the 
simple modules associated to the vertices of the quiver.

To deal with the problem of effective computation of these resolutions, 
Theorem \ref{teo1} below gives sufficient conditions for a complex defined over these projective bimodules to be exact.
We will be, in consequence, able to prove that some complexes are resolutions without following the procedure prescribed in the proof of the existence theorem.

Briefly, we do the following: given an algebra $A=kQ/I$ we compute a bimodule resolution of $A$ from a reduction system $\mathcal{R}$ 
for $I$ which satifies a condition we denote $(\lozenge)$. We prove that such a reduction system always exists, but we also show in an example that 
it may not be the most convenient one. In particular the resolution obtained may not be minimal.

Applying our method we recover a well-known resolution of quantum complete
intersections, see for example \cite{BE} and \cite{BGMS}.
We also construct a short resolution for down-up algebras 
which allows us to
prove that a noetherian down-up algebra $A(\alpha,\beta,\gamma)$ 
is $3$-Calabi-Yau if 
and only if $\beta=-1$.

\medskip

The contents  of the article are as follows. 
In Section \ref{s:Preliminaries} we fix notations and prove some preliminary results.
In Section \ref{s:Proj_modules} we deal with ambiguities.
In Section \ref{s:Resolution} we state the main theorems of this article, namely Theorem \ref{teo1} and Theorem \ref{teo2}, after proving some results on orders and differentials.
Section \ref{s:Proofs} is devoted to the proofs of these theorems; it contains several technical lemmas. 
In Section \ref{s:Morphisms} we construct explicitely the differentials in low degrees and
in Section \ref{s:Examples} we give several applications of our results.

Finally, in Section \ref{s:Final remarks} we give sufficient conditions on the reduction system for minimality of any 
resolution obtained from it. We also prove that in case $A$ is graded by the length of paths, and it has a reduction system satisfying the conditions
requiered for minimality of the resolution,
then $A$ is $N$-Koszul if and only if the associated monomial algebra $A_S$ is $N$-Koszul. 

\medskip

We have just seen a recent preprint by Guiraud, Hoffbeck and Malbos \cite{GHM} where they construct a resolution that may be related to ours.


\medskip

We are deeply indebted to Mariano Su\'arez-Alvarez and Eduardo Marcos for their help in improving this article.
We also thank Roland Berger, Quimey Vivas and Pablo Zadunaisky for 
discussions and comments.


\section{Preliminaries}

\label{s:Preliminaries}
In this section we give some definitions, present some basic constructions and we also prove
results that are necessary in the sequel.

Let $k$ be a field and $Q$ a quiver with a finite set of vertices. Given $n\in 
\mathbb{N}$, $Q_n$ denotes the set of paths of length $n$ in $Q$ 
and $Q_{\geq n}$ the set of paths of length at least $n$, that is, $Q_{\geq n}=\bigcup_{i\geq n}Q_i$. Whenever $c\in Q_n$, we will write 
$|c|=n$. If $a,b,p, q \in Q_{\ge 0}$ are such that $q=apb$, we say that $p$ is
a {\em divisor} of $q$; if, moreover, $a=1$, we say that $p$ is a {\em left 
divisor} of $q$ and analogously for $b=1$ and {\em right divisor}. We denote $\mathsf{t},\mathsf{s}:Q_1\to Q_0$
the usual source and target functions. Given $s\in Q_{\geq0}$ and a finite sum $f=\sum_i\lambda_ic_i\in kQ$
such that $c_i\in Q_{\geq0}$ and $\mathsf{t}(s)=\mathsf{t}(c_i)$, $\mathsf{s}(s) = \mathsf{s}(c_i)$
for all $i$, we say that $f$ is \textit{parallel} to $s$.
Let $E\colonequals kQ_0$ be the subalgebra of the path algebra generated by
the vertices of $Q$.

Given a set $X$ and a ring $R$, we denote $\llrr{X}_R$ the left 
$R$-module freely spanned by $X$.

Let $I$ be a two sided ideal of $kQ$, $A=kQ/I$ and $\pi:kQ\to A$ the 
canonical projection. We assume that $\pi(Q_0\cup Q_1)$ is linearly independent.

We recall some terminology from \cite{B} that we will use.
A set of pairs $\mathcal R=\{(s_i,f_i)\}_{i\in\Gamma}$ where $s_i\in Q_{\geq0}$, 
$f_i\in kQ$  is called a \textit{reduction system}.
We will always assume that a reduction system $\mathcal R=\{(s_i,f_i)\}_{i\in\Gamma}$
satisfies the following conditions 
\begin{itemize}
 \item for all $i$, $f_i$ is parallel to $s_i$ and $f_i\neq s_i$ .
 \item $s_i$ does not divide $s_j$ for $i\neq j$.
\end{itemize}
Given $(s,f)\in \mathcal R$ and $a,c\in Q_{\geq0}$ such that $asc\neq0$ in $kQ$, we will call the
triple $(a,s,c)$ a \textit{basic reduction} and write it $r_{a,s,c}$. 
Note that $r_{a,s,c}$ determines an $E$-bimodule endomorphism 
$r_{a,s,c}:kQ\to kQ$ such that $r_{a,s,c}(asc)=afc$ and 
$r_{a,s,c}(q)=q$ for all $q\neq asc$.

A \textit{reduction} is an $n$-tuple $(r_n,\dots,r_1)$ where $n\in\NN$ and $r_i$ is a 
basic reduction for $1\leq i\leq n$. As before, a reduction $r=(r_n,\dots,r_1)$
determines an $E$-bimodule endomorphism of $kQ$, the composition of the endomorphisms corresponding to
the basic reductions $r_n,\dots,r_1$.

An element $x\in kQ$ is said to be \textit{irreducible} for $\mathcal R$ if $r(x)=x$ for all basic 
reductions $r$. We will omit mentioning the reduction system  whenever it is clear from the context.
A path $p\in Q_{\geq0}$ will be called \textit{reduction-finite} if for any
infinite sequence of basic reductions $(r_i)_{i\in\NN}$, there exists $n_0\in\NN$
such that for all $n\geq n_0$, $r_n\circ\cdots\circ r_1(p)=
r_{n_0}\circ\cdots\circ r_1(p)$. Moreover, the path $p$ will be called
\textit{reduction-unique} if it is reduction-finite and for any two 
reductions $r$ and $r'$ such that $r(p)$ and $r'(p)$ are both irreducible, the equality $r(p)=r'(p)$ holds.

\begin{definition}
 We say that a reduction system $\mathcal R$ satisfies condition $(\lozenge)$ for $I$ if
\begin{itemize}
 \item the ideal $I$ is equal to the two sided ideal generated by the set $\{s-f\}_{(s,f)\in\mathcal R}$, \item every path is reduction-unique and
 \item for each $(s,f)\in\mathcal R$, $f$ is irreducible. 
\end{itemize}
\end{definition}

The reason why we are interested in these reduction systems is the following lemma, which is a restatement of Bergman's Diamond Lemma.
\begin{lemma}
\label{lemma:bases}
 If the reduction system $\mathcal R$ satisfies $(\lozenge)$ for $I$, then the set $\B$ of irreducible paths satisfies the following properties,
  \begin{enumerate}[(i)]
  \item $\B$ is closed under divisors,
  \label{condicionbase1}
  \item $\pi(b)\neq\pi(b')$ for all $b,b'\in\B$ with $b\neq b'$,
  \label{condicionbase2}
  \item $\{\pi(b):b\in\B\}$ is a basis of $A$.
  \label{condicionbase3}
  \end{enumerate}
\end{lemma}

\begin{remark}
\label{remark:beta}
 In view of Lemma \ref{lemma:bases}, we can define a $k$-linear map $i: A \to kQ$ such that  be 
 $i(\pi (b))) = b$ for all $b\in  \B$. We denote by $\beta: kQ \to kQ$ the 
 composition $i\circ \pi$. Notice that if $p$ is a path and $r$ is a reduction such that $r(p)$ is irreducible, then $r(p)=\beta(p)$. 
 In the bibliography, $\beta(p)$ is sometimes called the normal form of $p$.
\end{remark}

\begin{definition}
If $\mathcal R$ is a reduction system satisfying $(\lozenge)$ for $I$, we define $S\colonequals\{s\in Q_{\geq0}: (s,f)\in\mathcal R\mbox{ for some }f\in kQ\}$.
\end{definition}
\begin{remark}
\label{remark:bprops_sb}
 Notice that:
 \begin{enumerate}
  \item $S$ is equal to the set $\{p\in Q_{\geq0}: p\notin\B\mbox{ and }p'\in\B\mbox{ for all proper divisors }p'\mbox{ of }p\}$.
   \item If $s$ and $s'$ are elements of $S$ such that $s$ divides $s'$, 
 then $s=s'$.
 \item Given $q\in Q_{\geq0}$, $q$ is irreducible if and only if there exists
 no $p\in S$ such that $p$ divides $q$.
 \end{enumerate}
\end{remark}

\begin{definition}
 Given a path $p$ and $q=\sum_{i=1}^n\lambda_ic_i\in kQ$ with $\lambda_1,\dots,\lambda_n\in k^\times$ and $c_1,\dots,c_n\in Q_{\geq0}$, we write $p\in q$ if $p=c_i$ for some $i$, or, in other words, when $p$ is in the support of $q$.
\end{definition}

 Given $p,q\in Q_{\geq0}$ we write $q\leadsto p$ if there exist $n\in\NN$, basic reductions $r_1,\dots,r_n$ and paths $p_1,\dots p_n$ such that $p_1=q$, $p_n=p$, and for all $i=1,\dots,n-1$, $p_{i+1}\in r_i(p_i)$.

\begin{lemma}
 \label{lemma:leadsto_es_orden}
 Suppose that every path is reduction-finite with respect to $\mathcal R$. 
 \begin{enumerate}[(i)]
  \item If $p$ is a path and $r$ a basic reduction such that $p\in r(p)$, 
then $r(p)=p$.
  \item The binary relation $\leadsto$ is an order on the set $Q_{\geq0}$ which is compatible with concatenation, that is, $\leadsto$ satisfies that $q\leadsto p$ implies $aqc\leadsto apc$ for all $a,c\in Q_{\geq0}$ such that $apc\neq0$ in~$kQ$.
  
  \item The binary relation $\leadsto$ satisfies the descending chain condition.
 \end{enumerate}
\end{lemma}

\begin{proof}
\textit{(i)} 
The hypothesis means that $r(p) = \lambda p + x$ with $\lambda\in 
k^\times$ and $p\notin x$. If $x\neq0$ or $\lambda\neq1$, then $r$ acts 
nontrivially on $p$ and so it acts trivially on $x$. Since
the sequence of reductions $(r,r,\cdots)$ stabilizes when acting on $p$, 
there exists $k\in\NN$ such that $\lambda^kp+kx=r^k(p) = 
r^{k+1}(p)=\lambda^{k+1}p+(k+1)x$. As a consequence, $\lambda=1$ and $x=0$.

\textit{(ii)}
It is clear that $\leadsto$ is a transitive and reflexive relation and that 
it is compatible with concatenation. Let us suppose that it is not 
antisymmetric, so that there exist $n\in\NN$, paths $p_1,\dots,p_{n+1}$ and 
basic reductions $r_1,\dots,r_n$ such that $p_{i+1}\in r_{i}(p_{i})$ for 
$1\leq i\leq n$ and $p_{n+1}=p_1$. Suppose that $n$ is minimal. There exist 
$x_1,\dots,x_n\in kQ$ and $\lambda_1,\dots,\lambda_n\in k^\times$ such that 
$r_i(p_i) = \lambda_{i}p_{i+1} + x_{i}$ with $p_{i+1}{\notin x_i}$. Notice 
that since $n$ is minimal, $r_i(p_i)\neq p_i$ and then $r_i$ acts trivially 
on every path different from $p_i$, for all $i$.

Let us see that 
\[
 \text{ $p_i\notin x_j$ for all $i\neq j$}.
\]

Since the sequence $p_1,\dots,p_{n+1}=p_1$ is cyclic, it is enough 
to prove that $p_1\notin x_j$ for all $j$. Suppose that $p_1\in x_j$ for 
some $j\in\{1,\dots,n\}$. Since $p_{i+1}\notin x_i$ for all $i$ and 
$p_{n+1}=p_1$, it follows that $j\neq n$, and by part (i), 
$j\neq1$. Let 
$u_k=p_{k}$ and $t_k=r_{k}$ for $1\leq k\leq j$ and 
$u_{j+1}=p_1$.
Notice that $u_{k+1}\in t_k(u_k)$ for $1\leq k\leq j$ and $u_{j+1}=u_1$. 
Since $j<n$ this contradicts the choice of $n$. It follows that 
\[
 \text{ $p_i\notin x_j$ for all $i,j$}.
\]
One can easily check that this implies $r_n\circ \cdots \circ 
r_1(p_1)=\lambda p_1 + x$ with $p_i\notin x$ for all $i$. Now, define 
inductively for $i>n$, $r_i\colonequals r_{i-n}$. The sequence 
$(r_i)_{i\in\NN}$ acting on $p_1$ never stabilizes, which contradicts the 
reduction-finiteness of the reduction system $\mathcal R$.

\textit{(iii)}
Suppose not, so that there is a sequence $(p_i)_{i\in\NN}$ of paths and a sequence of basic reductions $(t_i)_{i\in\NN}$ such that $p_{i+1}\in t_i(p_i)$. Since $\leadsto$ is an antisymmetric relation, $p_i\neq p_j$ if $i\neq j$.
 
 Let $i_1=1$. Suppose that that we have constructed $i_1,\dots,i_k$ such that
 $i_1<\cdots<i_k$, $p_{i_k}\in t_{i_{k-1}}\circ\cdots\circ t_1(p_1)$ and
 $p_j\notin t_{i_{k-1}}\circ\cdots\circ t_1(p_1)$ for all $j>i_k$. Set $X_k=\{i>i_k: p_i\in
 t_{i_k}\circ\cdots\circ t_{i_1}(p_1)\}$. By the inductive hypothesis, there is $x\in kQ$ and 
 $\lambda\in k^\times$ such that $t_{i_{k-1}}\circ\cdots\circ t_{i_1}(p_1) = \lambda p_{i_k}+x$ with
 $p_{i_k}\notin x$. Since we also know that $p_{i_k+1}\in t_{i_k}(p_{i_k})$, and 
 $p_{i_k+1}\notin t_{i_{k-1}}\circ\cdots\circ t_{i_1}(p_1)$ 
 it follows that $p_{i_k+1}\in t_{i_k}(p_{i_k})+x$. Also, 
$t_{i_k}\circ\cdots\circ t_{i_1}(p_1)=\lambda 
t_{i_k}(p_{i_k})+t_{i_k}(x)=\lambda t_{i_k}(p_{i_k}) + x$, so $p_{i_k+1}\in 
t_{i_k}\circ\cdots\circ t_{i_1}(p_1)$.
 Therefore $X_k$ is not empty. We may define $i_{k+1}=\max X_k$, because $X_k$ is a finite set.
 
 This procedure constructs inductively a strictly increasing sequence of indices $(i_k)_{k\in\NN}$ with $p_{i_k}\in\tilde p_{i_k}\colonequals t_{i_{k-1}}\circ\cdots\circ t_{i_1}(p_1)$ for all $k\in\NN$. The set $\{t_{i_{k-1}}\circ\cdots\circ t_{i_1}(p_1):k\in\NN\}$ is therefore infinite. This contradicts the reduction-finiteness of $\mathcal R$.
\end{proof}

 The converse to Lemma \ref{lemma:leadsto_es_orden} also holds, that is,
 if $\mathcal R$ is a reduction system for which $\leadsto$ is a partial order
 satisfying the descending chain condition, then every path is reduction-finite.
 In other words, the order $\leadsto$ captures most of the properties we require
 $\mathcal R$ to verify, and it will be important in the next sections.

 \medskip
 
The following characterization of the relation $\leadsto$ is very useful in practice.

\begin{lemma}
\label{lemma:caract_orden_leadsto}
 If $p,q$ are paths, then $q\leadsto p$ if and only if $p=q$ or there exists a reduction $t$ such that $p\in t(q)$.
\end{lemma}

\begin{proof}First we prove the necessity of the condition.
 Let $n\in\NN$, $r_1,\dots,r_n$ and $p_1,\dots,p_n$ be as in the definition of $\leadsto$, and suppose that
 $n$ is minimal. Let $\tilde p_1=p_1$ and for each $i=1,\dots,n-1$ put $\tilde p_{i+1}=r_i(\tilde p_i)$. Notice that the minimality implies that $r_i(p_i)\neq p_i$. 
 Let us first show that 
 \begin{equation}
 \label{tilde}
  \text{ if $i>j$ then $p_i\notin \tilde p_j.$}
 \end{equation}
 Suppose otherwise and let $(i,j)$ be a counterexample with $j$ minimal. 
 We will prove that in this situation, $p_l\in \tilde 
p_l$ for all $l<j$. We proceed by induction on $l$. By 
definition, $p_1\in \tilde p_1$. Suppose $1\leq l<j-1$ and $p_l\in \tilde 
p_l$. Then we have
 $p_{l+1}\in r_l(p_l)$ and, since $l<j$, $p_{l+1}\notin \tilde p_l$.
 Write $\tilde p_l = \lambda p_l + x$ with $x\in kQ$ and $p_l\notin x$. 
Since $r_l$ acts nontrivially on $p_l$, it acts trivially on $x$; it 
follows that $r_l(\tilde p_l) = \lambda r_l(p_l) + x$ and so $p_{l+1}\in 
r_l(\tilde p_l)=\tilde p_{l+1}$. In particular $p_{j-1}\in\tilde p_{j-1}$. 
Since $p_i\notin \tilde p_{j-1}$ and $p_i\in\tilde p_j$, we must have 
$p_i\in r_{j-1}(p_{j-1})$.

 Now, let $m=n+j-i$, $t_k=r_k$ and $u_k=p_k$ if $k\leq j-1$, and 
$t_k=r_{i+k-j}$ and $u_k=p_{i+k-j}$ if~$j\leq k\leq m$. One can check that 
$u_1=q$, $u_{n+j-i}=p$ and
 that $u_{k+1}\in t_k(u_k)$ for all $k=1,\dots,m-1$. Since $m<n$ this contradicts the choice 
 of $n$. We thus conclude that \eqref{tilde} holds.
 
 We can use the same inductive argument as before to prove 
that $p_i\in\tilde p_i$ for all $1\leq i\leq n$. 
 Denoting $t=(r_n,\dots,r_1)$,  observe that $p \in t(q)$.
 
 Let us now prove the converse. Let $t=(t_m,\dots,t_1)$ be a reduction 
such that $p\in t(q)$ and $m$ is minimal, and let us proceed by induction 
on $m$. Notice that if $m=1$ there is nothing to prove. If $t_i$ is the 
basic reduction $r_{a_i,s_i,c_i}$, let $p_i=a_is_ic_i$. Using the same 
ideas as above one can show that
 \[
  \begin{minipage}{0.6\displaywidth}
   if $u\neq q$ and $u\notin t_i(p_i)$ for each $1\leq i\leq m$, then $u\notin t_l\circ\cdots\circ t_1(q)$ for each $0\leq l\leq m$.
  \end{minipage}
 \]
 Since $p\in t(q)$ either $p=q$ or there exists $i\in\{1,\dots,m\}$ such that $p\in t_i(p_i)$. In the first case $q\leadsto p$. In the second case, we know that $p_i\leadsto p$ and we need 
 to prove that $q\leadsto p_i$. Since $m$ is minimal, $t_i(t_{i-1}\circ\cdots\circ t_1(q))\neq t_{i-1}\circ\cdots\circ t_1(q)$ and then $p_i\in t_{i-1}\circ\cdots\circ t_1(q)$. The result now follows
 by induction because $i-1<m$.
\end{proof}

\medskip

\begin{proposition}
\label{prop:existencia_sr}
 If $I\subseteq kQ$ is an ideal, then there exists a reduction system $\mathcal R$ which satisfies condition $(\lozenge)$ for $I$.
\end{proposition}

We will prove this result by putting together a series of lemmas.

Let $\leq$ be a well-order on the set $Q_0\cup Q_1$ such that $e<\alpha$ for all $e\in Q_0$ and $\alpha\in Q_1$. Let $\omega:Q_1\to\NN$ be a function and extend it to $Q_{\geq0}$ defining $\omega(e)=0$ for all $e\in Q_0$ and $\omega(c_n\cdots c_1)=\sum_{i=1}^n\omega(c_i)$ if $c_i\in Q_1$ and $c_n\cdots c_1$ is a path.
Given $c,d\in Q_{\geq0}$ we write that $c\leq_\omega d$ if
\begin{itemize}
 \item $\omega(c)<\omega(d)$, or
 \item $c,d\in Q_0$ and $c\leq d$, or
 \item $\omega(c)=\omega(d)$, $c=c_n\cdots c_1,d=d_m\cdots d_1\in Q_{\geq1}$ and there exists 
   $j\leq\mathsf{min}(|c|,|d|)$ such that $c_i=d_i$ for all 
   $\in\{1,\dots,j-1\}$ and $c_j<d_j$.
\end{itemize}

Notice that the order $\leq_\omega$ is in fact the {\em deglex} order with weight $\omega$, and it has the following two properties:

\begin{enumerate}[(i)]
 \label{lemma:ordenes}
  \item If $p,q\in Q_{\geq0}$ and $p\leq_\omega q$, then $cpd\leq_\omega cqd$ for all 
  $c,d\in Q_{\geq0}$ such that $cpd\neq0$ and $cqd\neq0$ in $kQ$.
  \label{lemma:ordenes_item1}
  \item For all $q\in Q_{\geq0}$ the set $\{p\in Q_{\geq0} : p\leq_\omega q\}$ 
  is finite.
  \label{lemma:ordenes_item2}
\end{enumerate}

 It is straightforward to prove the first claim.
 For the second one, let $\{c^i\}_{i\in\NN}$ be a sequence in $Q_{\geq0}$ such that $c^{i+1}\leq_\omega c^i$ for all $i$.
If $c^i\in Q_0$ for some $i$, then it is evident that the sequence stabilizes, so let us suppose that $\{c^i\}_{i\in\NN}$ is contained in $Q_{\geq1}$ 
and $c^{i+1}<_\omega c^i$ for all $i\in\NN$. 
Since $(\omega(c^i))_{i\in\NN}$ is a decreasing sequence of natural 
numbers, it must stabilize, so we may also suppose that 
$\omega(c^i)=\omega(c^j)$ for all $i,j$ and that the lengths of the paths 
are bounded above by some $M\in\NN$. 
By definition of $\leq_\omega$, we know that the sequence of first 
arrows of elements of $\{c^i\}_{i\in\NN}$ 
forms a decreasing sequence in $(Q_1,\leq)$, which must stabilize because $(Q_1,\leq)$ is well-ordered. Let $N\in\NN$ be such that the first arrow of
$c^i$ equals the first arrow of $c^j$ for all $i,j\geq N$. If $c^i = c_{n_i}^i\cdots c_1^i$, and we denote $c'^i=c_{n_i}^i\cdots c_2^i$, then 
$\{c'^i\}_{i\geq N}$ is a decreasing sequence in $(Q_{\geq0},\leq_\omega)$ with $|c'^i|=M-1$ for all $i$. Iterating this process we arrive to a contradiction.


\medskip

\begin{definition}
 Consider as before a well-order $\leq$ on $Q_0\cup Q_1$ and 
$\omega:Q_1\to\NN$, and $\leq_\omega$ be constructed from 
them. If $p\in kQ$ and $p=\sum_{i=1}^n\lambda_ic_i$ with $\lambda_i\in 
k^\times$,
 $c_i\in Q_{\geq 0}$ and $c_i<_{\omega} c_1$ for all $i\neq1$, we write $\mathsf{tip}(p)$
 for $c_1$. If $X\subseteq kQ$, we let $\mathsf{tip}(X)\colonequals\{\mathsf{tip}(x): x\in X\setminus\{0\}\}$.
\end{definition}

Consider the set
\[
S\colonequals \mathsf{Mintip}(I) = \{p\in\mathsf{tip}(I): p'\notin\mathsf{tip}(I)\mbox{ for all 
proper divisors }p'\mbox{ of }p\}.
\]
Notice that if $s$ and $s'$ both belong to $S$ and $s\neq s'$, then $s$ does not divide $s'$.
For each $s\in S$, choose $f_s\in kQ$ such
that $s-f_s\in I$, $f_s<_\omega s$ and $f_s$ is parallel to $s$.

Describing the set $\mathsf{tip}(I)$ is not easy in general. We comment on this problem at the beginning of the last section, where we compute examples.

\begin{lemma}
\label{lemma:existencia-sr1}
 Let $\leq_\omega$ and $S$ be as before. The ideal $I$ equals the two sided ideal generated by the set $\{s-f_s\}_{s\in S}$, which we will denote by $\llrr{s-f_s}_{s\in S}$.
\end{lemma}
\begin{proof}
 It is clear that $\llrr{s-f_s}_{s\in S}$ is contained in $I$.
 Choose $x=\sum_{i=1}^n\lambda_ic_i\in I$ with $\lambda_i\in k^\times$ and 
 $c_i\in Q_{\geq0}$. We may suppose that $c_1=\mathsf{tip}(x)$, so that $c_1\in\mathsf{tip}(I)$. There is a divisor
 $s$ of $c_1$ such that $s\in\mathsf{tip}(I)$ and $s'\notin\mathsf{tip}(I)$ 
 for all proper divisor $s'$ of $s$ and $s\in S$ by definition of $S$. Let
 $a,c\in Q_{\geq0}$ with $asc=c_1$. 
 
 Define $x'\colonequals af_sc +\sum_{i=2}^n\lambda_ic_i$. We have 
 $x = \lambda_1c_1 + \sum_{i=2}^n\lambda_ic_i = \lambda_1a(s-f_s)c + x'$, 
so that $x'\in I$ and, by property (\ref{lemma:ordenes_item1}) of the 
order $\leq_\omega$, we see that 
 $c_1>\mathsf{tip}(x')$. We can apply this procedure again to $x'$ and iterate: the process
 will stop by property (\ref{lemma:ordenes_item2}) and we conclude that $x\in\llrr{s-f_s}_{s\in S}$.	
\end{proof}

\begin{lemma}
\label{lemma:existencia-sr2}
 Let $\leq_\omega$ and $S$ be as before. The set $\mathcal R\colonequals\{(s,f_s)\}_{s\in S}$ is a reduction system 
 such that every path is reduction-unique.
\end{lemma}

\begin{proof}
 Since $s>_\omega\mathsf{tip}(f_s)$ for all $s\in S$, properties  
 (\ref{lemma:ordenes_item1}) and (\ref{lemma:ordenes_item2}) guarantee 
that every path is reduction-finite. We need to prove that every path is 
reduction-unique. 
 Recall that $\pi$ is the canonical projection
 $kQ\to kQ/I$. Let $p$ be a path. Since $I=\llrr{s-f_s}_{s\in S}$, we see that 
 $\pi(r(p)) = \pi(p)$ for any reduction $r$. Let
 $r$ and $t$ be reductions such that $r(p)$ and $t(p)$ are both irreducible. Clearly,
 $\pi(r(p)-t(p)) = \pi(p)-\pi(p)=0$, so that $r(p)-t(p)\in I$.
 If this difference is not zero, then the path $d=\mathsf{tip}(r(p)-t(p))$ 
can be written as $d=asc$ with $a,c$ paths and $s\in S$. It follows that 
the reduction $r_{a,s,c}$ acts nontrivially either on $r(p)$ or 
 on $t(p)$, and this is a contradiction.
\end{proof}

This lemma implies that for each $s\in S$, 
there exists a reduction $r$ and an irreducible element $f_s'$ such that 
$r(f_s)=f_s'$. Consider the reduction system 
$\mathcal R'\colonequals\{(s,f_s'):s\in S\}$. The set of
irreducible paths for $\mathcal R$ clearly coincides with the set of irreducible paths
for $\mathcal R'$ and, since $\pi(s-f_s')=\pi(s-f_s)=0$, we have that 
$\llrr{s-f_s'}_{s\in S}\subseteq I$. From Bergman's Diamond Lemma it follows
that $I=\llrr{s-f_s'}_{s\in S}$.
We can conclude that the reduction system $\mathcal R'$ satisfies condition $(\lozenge)$, thereby proving
Proposition \ref{prop:existencia_sr}.

\medskip 

It is important to emphasize that different choices of orders on $Q_0\cup Q_1$ and of weights $\omega$ will
give very different reduction systems, some of which will better suit 
our purposes than others. Moreover, there are reduction systems which cannot 
be obtained by this procedure, as the following example shows.

\begin{example}
\label{example:otroorden}
Consider the algebra
\[
  A=k\langle x,y,z \rangle /(x^3+y^3+z^3-xyz)
\]
and let $\mathcal R=\{(xyz,x^3+y^3+z^3)\}$. Clearly this reduction system does not 
come from a monomial order and neither from a monomial order with weights. It is not entirely evident but this
reduction system satisfies $(\lozenge)$.
\end{example}

\medskip 

Finally, we define a relation $\preceq$ on the set $k^\times Q_{\geq0}\colonequals\{\lambda p: \lambda\in k^\times, p\in Q_{\geq0}\}\cup\{0\}$ as the least reflexive and transitive relation such that $\lambda p\preceq \mu q$ whenever there exists a reduction $r$ such that $r(\mu q)= \lambda p + x$ with $p\notin x$. We state $0\preceq \lambda p$ for all $\lambda p\in k^\times Q_{\geq0}$.

\begin{lemma}
 The binary relation $\preceq$ is an order satisfying the descending chain condition and it is compatible with concatenation.
\end{lemma}

\begin{proof}
 The second claim is clear. In order to prove the first claim, let us first 
prove that if $p\in Q_{\geq0}$ is such that
 there exists a reduction $r$ with $r(p) = \lambda p +x$ and $p\notin x$, then $\lambda =1$ and $x=0$. Suppose 
 not. For $r$ a basic reduction, this has already been 
done in Lemma 
\ref{lemma:leadsto_es_orden}.
 If $r$ is not basic, then $r=(r_n,\dots,r_1)$ with $r_i$ basic 
and $n\geq2$. Let $r'=(r_n,\dots,r_2)$. Since $p\in r(p)=r'(r_1(p))$, 
there exists $p_1\in r_1(p)$ such that $p\in r'(p_1)$. By the previous 
case, we obtain that $p\notin r_1(p)$, so $p\neq p_1$. As a consequence 
of Lemma \ref{lemma:caract_orden_leadsto}, we know that $p\leadsto 
p_1$ since $p_1\in r_1(p)$ and that $p_1\leadsto p$ since $p\in r'(p_1)$. 
This contradicts the antisymmetry of $\leadsto$.
 
 It is an immediate consequence of the previous fact that given a path $p$ and a reduction $t$,
\begin{equation}
\label{mariano}
 \text{if $t(\lambda_1p)=\lambda_2p + x$ with $p\notin x$, then 
$\lambda_1=\lambda_2$.}
\end{equation}
Let $\lambda_1,\dots,\lambda_{n+1}\in k^\times$, $p_1,\dots,p_{n+1}\in 
Q_{\geq0}$, $x_1,\dots,x_n\in kQ$ and 
reductions $t_1,\dots,t_n$ be such that 
$t_i(\lambda_ip_i)=\lambda_{i+1}p_{i+1} + x_i$, $p_{i+1}\notin x_i$ 
and $\lambda_{n+1} p_{n+1} = \lambda_1p_1$. This implies that $p_i\leadsto 
p_{i+1}$ for each $1\leq i\leq n$ 
and $p_{n+1}=p_1$. Since $\leadsto$ is antisymmetric, it follows that 
$p_i=p_1$ for all $i$ and 
\eqref{mariano} implies that $\lambda_i=\lambda_1$ for all $i$. We thus see 
that $\preceq$ is antisymmetric.

Let now $(\lambda_ip_i)_{i\in\NN}$ be a sequence in $k^\times Q_{\geq0}$ 
and $(t_i)_{i\in\NN}$ a sequence of
reductions such that $t_i(\lambda_ip_i) =\lambda_{i+1}p_{i+1} + x_i$ with 
$p_{i+1}\notin x_i$. Then $p_i\leadsto p_{i+1}$ for all $i$ and since 
$\leadsto$ satisfies the descending chain condition there exists $i_0$ such 
that $p_i=p_{i_0}$ for all $i\geq i_0$. Observation \eqref{mariano} implies 
then that $\lambda_i=\lambda_{i_0}$ for all $i\geq i_0$, so that the 
sequence $(\lambda_ip_i)_{i\in\NN}$ stabilizes.
\end{proof}

If $x=\sum_{i=1}^n\lambda_ip_i\in kQ$ with $\lambda_i\in k^\times$ and 
$\lambda p$ belongs to $k^\times Q_{\geq0}$, we write $x\preceq \lambda p$ 
if $\lambda_ip_i\preceq \lambda p$ for all $i$. If in addition $x\neq 
\lambda p$ we also write $x\prec p$. The following simple fact is key to 
proving everything that follows.

\begin{corollary}
Given a path $p$, its normal form $\beta(p)$ is such that $\beta(p)\preceq p$. Moreover, $\beta(p)\prec p$ if and 
only if $p\notin\B$.
\end{corollary}
\begin{proof}
 There is a reduction $r$ such that 
$\beta(p)=r(p)=\sum_{i=1}^n\lambda_ip_i$. It is clear that 
$\lambda_ip_i\preceq p$ for all $i$, so that $\beta(p)\preceq p$. The last 
claim follows from the fact that $\beta(p)=p$ if and only if $p\in\B$.
\end{proof}

\section{Ambiguities}
\label{s:Proj_modules}

Given an algebra $A=kQ/I$ and a reduction system $\mathcal{R}$ satisfying $(\lozenge)$ for $I$, there is a 
monomial algebra associated to $A$ defined as $A_{S}\colonequals 
kQ/\langle S\rangle$ and equipped with the canonical projection 
$\pi':kQ\to A_{S}$. The set $\pi'(\B)$ is a $k$-basis of 
$A_{S}$. 
The algebra $A_S$ is a generalization of the algebra $A_{mon}$ defined in \cite{GM}: in that article, the order
is necessarily monomial.

From now on we fix the reduction system $\mathcal R$ satisfying condition $(\lozenge)$. Notice that in this situation we can suppose without loss of generality, that 
$S\subseteq Q_{\geq2}$.

The family of modules $\{{\mathcal{P}_i}\}_{i\ge 0}$ appearing in the resolution of 
$A$ as $A$-bimodule will be in bijection with those appearing in Bardzell's resolution
of the monomial algebra $A_{S}$.
More precisely, we will define $E$-bimodules $k\A_i$ for $i\ge -1$, such 
that the former will be $\{A\otimes_E k\A_i \otimes_E A\}_{i\geq-1}$
while the latter will be $\{A_{S}\otimes_E k\A_i \otimes_E 
A_{S}\}_{i\geq-1}$.
The resolution will start as usual: $\A_{-1} = Q_0$, $\A_0=Q_1$ and $\A_1=S$.

For $n\geq2$, $\A_n$ will be the set of
$n$-ambiguities of $\mathcal R$. We will next recall the definition of $n$-ambiguity
-- or $n$-chain according to the terminology used in \cite{Sk}, \cite{An}, 
\cite{AG} and 
to Bardzell's \cite{Ba} \textit{associated sequences of paths}, 
and we will take into account that the sets of left
$n$-ambiguities and right $n$-ambiguities coincide. This fact is proved in
\cite{Ba} and also in \cite{Sk}. 
 See  \cite{GZ} too.

\begin{definition} Given $n\geq2$ and $p\in Q_{\geq0}$,
 \begin{enumerate}
  \item the path $p$ is a {\em left} $n$-{\em ambiguity} if there exist $u_0\in Q_1$, 
   $u_1,\dots,u_n$ irreducible paths such that
    \begin{enumerate}[(i)]
      \item $p=u_0u_1\cdots u_n$,
      \label{def:nambiguedadiz1}
      \item for all $i$, $u_iu_{i+1}$ is reducible but $u_id$ is irreducible
	for any proper left divisor $d$ of $u_{i+1}$.
      \label{def:nambiguedadiz2}
    \end{enumerate}
   \item the path $p$ is a {\em right} $n$-{\em ambiguity} if there exist $v_0\in Q_1$ and
   $v_1,\dots,v_n$ irreducible paths such that
    \begin{enumerate}[(i)]
     \item $p=v_n\cdots v_0$,
     \item for all $i$, $v_{i+1}v_i$ is reducible but $dv_i$ is irreducible
	for any proper right divisor of $v_{i+1}$.
    \end{enumerate}
 \end{enumerate}
\end{definition}

\begin{proposition}
 Let $n,m\in\NN$, $p\in Q_{\geq1}$. If $u_0,\hat u_0\in Q_1$ and $u_1,\dots
 u_n,\hat u_1,\dots,\hat u_n$ are paths in $Q$ such that both $u_0,\dots,u_n$
 and $\hat u_0,\dots,\hat u_n$ satisfy conditions (\ref{def:nambiguedadiz1})
 and (\ref{def:nambiguedadiz2}) of the previous definition for $p$, then 
 $n=m$ and $u_i=\hat u_i$ for all $i$, $0\leq i\leq n$.
\end{proposition}

\begin{proof}
 Suppose $n\leq m$. It is obvious that $u_0 = \hat u_0$, since both of them
 are arrows. Notice that $kQ = T_{kQ_0} kQ_1$, that is the free algebra 
 generated by $kQ_1$ over $kQ_0$, which implies that either $u_0u_1$ divides 
 $\hat u_0\hat u_1$ or $\hat u_0\hat u_1$ divides $u_0u_1$, and moreover 
 $u_0u_1,\hat u_0 \hat u_1\in\A_1= S$. Remark \ref{remark:bprops_sb} 
 says that $u_0u_1=\hat u_0\hat u_1$. Since $u_0=\hat u_0$, we must have 
 $u_1=\hat u_1$. By induction on $i$, let us suppose that $u_j=\hat u_j$ 
 for $j\leq i$. As a consequence, 
 $u_{i+1}\cdots u_n = \hat u_{i+1}\cdots\hat u_m$. 
 
 If $i+1=n$, this reads $u_n = \hat u_n\cdots\hat u_{m}$, and the fact that 
 $u_n$ is irreducible and $\hat u_j\hat u_{j+1}$ is reducible for all $j<m$
 implies that $m=n$ and $u_n=\hat u_n$. Instead, suppose that $i+1<n$. From
 the equality $u_{i+1}\cdots u_n=\hat u_{i+1}\cdots\hat u_m$ we deduce that
 there exists a path $d$ such that $u_{i+1}=\hat u_{i+1}d$ or $\hat u_{i+1}
 = u_{i+1}d$. If $u_{i+1}=\hat u_{i+1}d$ and $d\in Q_{\geq1}$, we can write
 $d=d_2d_1$ with $d_1\in Q_1$. The path $\hat u_{i+1}d_2$ is a proper left
 divisor of $u_{i+1}$ and by condition (\ref{def:nambiguedadiz2}) we obtain
 that $u_i\hat u_{i+1}d_2$ is irreducible. This is absurd since $u_i\hat 
 u_{i+1}d_2 = \hat u_i\hat u_{i+1}d_2$ by inductive hypothesis, and the 
 right hand  term is reducible by condition (\ref{def:nambiguedadiz2}). It
 follows that $d\in Q_0$ and then $u_{i+1}=\hat u_{i+1}$. The case where 
 $\hat u_{i+1} = u_{i+1}d$ is analogous.
\end{proof}

\begin{corollary}
 Given $n,m\geq-1$, $\A_n\cap\A_m=\emptyset$ if $n$ and $m$ are different.
\end{corollary}

Just to get a flavor of what $\A_n$ is, one may think about an element of 
$\A_n$ as a minimal proper superposition of $n$ elements of $S$.

\medskip

We end this section with a proposition that indicates how to compute ambiguities for a particular family of algebras.

\begin{proposition}
\label{cuadratico}
Suppose $S\subset Q_2$. For all $n\ge 1$, 
\[
\A_n = \{\alpha_0\dots\alpha_n \in Q_{n+1}: \alpha_i\in Q_1 \hbox{ for all } i \hbox{ and } \alpha_{i-1}\alpha_i \in S\}
\]
Moreover, given $p=\alpha_0\dots\alpha_n \in \A_n$, we can write $p$ as a left ambiguity choosing $u_i=\alpha_i$, for all $i$, and as a right ambiguity choosing
$v_i=\alpha_{n-i}$
\end{proposition}
\begin{proof}
We proceed by induction on $n$. If $n=1$ we know that $\A_1=S$ in which case
there is nothing to prove. Let $u_0\cdots u_nu_{n+1}\in\A_{n+1}$ and suppose 
that the result holds for all $p\in\A_n$.
Since $u_0\cdots u_n$ belongs to $\A_n$ we only have
to prove that $u_{n+1}\in Q_1$ and that $u_nu_{n+1}\in S$. We know that 
$u_n\in Q_1$, that $u_{n+1}$ is irreducible and that $u_nu_{n+1}$ is reducible. As a consequence, there exist $s\in S$ and $v\in Q_{\geq0}$ such that $u_nu_{n+1}=sv$. Moreover, $u_nd$ is irreducible for any proper left divisor $d$
of $u_{n+1}$, so the only possibility is  $v\in Q_0$. We conclude that $u_nu_{n+1}$ belongs to $S$. Since $S\subseteq Q_2$ and $u_n\in Q_1$, we deduce that $u_{n+1}\in Q_1$.
This proves that 
$\A_{n+1}\subseteq \{\alpha_0\cdots\alpha_n\in Q_{n+1}: \alpha_i\in Q_1 \hbox{ for all } i \hbox{ and }\alpha_{i-1}\alpha_i \in S\}$. 

The other inclusion is
clear.
\end{proof}



\section{The resolution}
\label{s:Resolution}

In this section our purpose is to construct bimodule resolutions of the 
algebra $A$. We achieve this in Theorems \ref{teo1} and \ref{teo2}: in the 
first one
we construct homotopy maps to prove that a given complex is exact, while in the second one we define differentials inductively.

We will make use of differentials of Bardzell's resolution for monomial 
algebras, so we begin this section by recalling them. Keeping 
the notations of the previous section, note that the $kQ$-
bimodule $kQ\otimes_E k\A_n\otimes_E kQ$ is  a $k$-vector space with basis
$\{a\otimes p\otimes c : a,c\in Q_{\geq0}, p\in\A_n, apc\neq0\mbox{ in }kQ\}$.

As we have already done for $A$, we define a $k$-linear map $i': A_S \to kQ$ such that  be 
$i'(\pi' (b))) = b$ for all $b\in  \B$, and we denote by $\beta': kQ \to kQ$ the 
composition $i'\circ \pi'$. 
 
Given $n\geq-1$, let us fix notation for the following $k$-linear maps:
\[
 \begin{aligned}
  &\pi_n\colonequals \pi\otimes\id_{k\A_n}\otimes\pi, &
    &\hskip0.5cm\pi_n'\colonequals \pi'\otimes\id_{k\A_n}\otimes\pi',\\
  &i_n\colonequals i\otimes\id_{k\A_n}\otimes i, &
    &\hskip0.5cmi_n'\colonequals i'\otimes\id_{k\A_n}\otimes i',\\
  &\beta_n \colonequals i_n\circ\pi_n, &
    &\hskip0.5cm\beta_n'\colonequals i_n'\circ\pi_n'.
 \end{aligned}
\]


Consider the following sequence of $kQ$-bimodules,
{\footnotesize
\[
 \xymatrix{\cdots\ar[r]^-{f_2}&kQ\otimes_Ek\A_1\otimes_E kQ\ar[r]^-{f_1}&
	    kQ\otimes_Ek\A_0\otimes_EA\ar[r]^-{f_0}&
	    kQ\otimes_EkQ\ar[r]^-{f_{-1}}\ar[d]^-{\cong}&
	    kQ\ar[r]&0\\
   & & & kQ\otimes_Ek\A_{-1}\otimes_EkQ&}
\]}
where 
\begin{enumerate}[(i)]
 \item $f_{-1}(a\otimes b) = ab$,
 \item if $n$ is even, $q\in\A_n$ and $q=u_0\cdots u_n=v_n\cdots v_0$ are
 respectively the factorizations of $q$ as left and right $n$-ambiguity,
  \[
    f_n(1\otimes q\otimes 1) = v_n\otimes v_{n-1}\cdots v_0\otimes 1 - 
				  1\otimes u_0\cdots u_{n-1}\otimes u_n,
  \]
 \item if $n$ es odd and $q\in\A_n$,
 \[
  f_n(1\otimes q\otimes 1) = 
      \underset{p\in\A_{n-1}}{\underset{apc=q}{\sum}}a\otimes p\otimes c.
 \]
\end{enumerate}
The maps $f_n$ induce, respectively, $A$-bimodule maps 
\[
\delta_n:A\otimes_E k\A_n\otimes_EA\to A\otimes_Ek\A_{n-1}\otimes_EA
\]
where
\[
\delta_n\colonequals \pi_{n-1}\circ f_n\circ i_n, 
\]
and $A_{S}$-bimodule maps 
\[\delta_n':A_{S}\otimes_Ek\A_n\otimes_EA_{S}\to
A_{S}\otimes_Ek\A_{n-1}\otimes_EA_{S}
\]
defined by 
\[
 \delta_n'\colonequals \pi_{n-1}'\circ f_n\circ i_n'.
\]
Observe that $\delta_{-1}$ and $\delta_{-1}'$ are respectively 
multiplication in $A$ and in $A_S$.

\medskip

The algebra $A_S$ is monomial. The following complex provides a projective
resolution of $A_{S}$ as $A_{S}$-bimodule \cite{Ba}:
{\footnotesize
\[
 \xymatrix{\cdots\ar[r]^-{\delta_2'}&A_S\otimes_Ek\A_1\otimes_EA_S\ar[r]^-{\delta_1'}&
	    A_S\otimes_Ek\A_0\otimes_EA_S\ar[r]^-{\delta_0'}&
	    A_S\otimes_EA_S\ar[r]^-{\delta_{-1}'}&
	    A_S\ar[r]&0}.
\]}
We will make use of the homotopy that Sk\"oldberg defined in \cite{Sk} 
when proving that this complex is exact.
We recall it, but we must stress that our signs differ from
the ones in \cite{Sk} due to the fact that he considers right modules, 
while we always work with left modules. 

Given $n\ge -1$, the morphism of $kQ-E$-bimodules $S_n$ 
is defined as follows. 

For $n=-1$, $S_{-1}:kQ\to kQ\otimes_Ek\A_{-1}\otimes_E kQ$ is the $kQ-E$-bimodule
map given by $S_{-1}(a)=a\otimes1$, for $a\in kQ$. 

For $n\in\NN_0$,
$S_n:kQ\otimes_Ek\A_{n-1}\otimes_E kQ\to kQ\otimes_Ek\A_n\otimes_EkQ$ is given
by
\[
 S_n(1\otimes q\otimes b) = 
 (-1)^{n+1}\underset{p\in\A_n}{\sum_{apc=qb}}a\otimes p\otimes c.
\]
Let $s_n'=\pi_n'\circ S_n\circ i_{n-1}'$.
The family of maps $\{s_n'\}_{n\geq-1}$ verifies the
equalities
\[
s_n'\circ\delta_n' + \delta_{n-1}'\circ s_{n-1}'= 
  \id_{A_S\otimes_Ek\A_n\otimes_EA_S} \hbox{ for } n\ge 0 \hbox{ and }s_{-1}'\circ\delta_{-1}'= 
  \id_{A_S\otimes_Ek\A_{-1}\otimes_EA_S}.
\]
  
\medskip
  
Next we define some sets that will be useful in the sequel.
For any $n\geq -1$ and $\mu q\in k^\times Q_{\geq0}$, consider the following 
subsets of $kQ\otimes_Ek\A_n\otimes_EkQ$: 
 \begin{itemize}
 \item $\LL_n^\preceq(\mu q) \colonequals 
 \{\lambda a\otimes p\otimes c: a,c\in Q_{\geq0}, p\in\A_n, \lambda apc \preceq \mu q\},$
 \item $\LL_n^\prec(\mu q) \colonequals 
 \{\lambda a\otimes p\otimes c: a,c\in Q_{\geq0}, p\in\A_n,\lambda apc \prec \mu q\},$
 \end{itemize}
and the following subsets of $A\otimes_Ek\A_n\otimes_EA$:
\begin{itemize}
  \item $\overline\LL_n^\preceq(\mu q) \colonequals
  \{\lambda\pi(b)\otimes p\otimes\pi(b'): b,b'\in\B, p\in\A_n, \lambda bpb'\preceq \mu q\}$,
  \item $\overline\LL_n^\prec(\mu q) \colonequals
  \{\lambda\pi(b)\otimes p\otimes\pi(b'): b,b'\in\B, p\in\A_n, \lambda bpb'\prec \mu q\}$.
\end{itemize}

\begin{remark}
\label{remark1}
 We observe that 
\[
  \begin{aligned}
  f_{n+1}(x)\in\,&\llrr{\LL_{n}^\preceq(\mu q)}_\ZZ, & &\mbox{ for all }
  x\in\LL_{n+1}^\preceq(\mu q)\mbox{, and }\\
  S_n(x)\in\,&\llrr{\LL_n^\preceq(\mu q)}_\ZZ, & &\mbox{ for all }
  x\in\LL_{n-1}^\preceq(\mu q).
 \end{aligned}
\]
Moreover, the only possible coefficients appearing in the linear combinations are $+1$ and $-1$.
\end{remark}

We will now state the main theorems. Recall that our aim is to 
construct, for non necessarily monomial algebras, a bimodule resolution 
starting from a related monomial algebra. The first theorem 
says that if the difference between its differentials and the monomial differentials
can be ``controlled'', then we will actually obtain an exact complex. The 
second theorem says that it is possible to construct the differentials.

\begin{theorem}
\label{teo1}
 Set $d_{-1}\colonequals\delta_{-1}$ and $d_0\colonequals\delta_0$. Given $N\in\NN_0$ and morphisms of $A$-bimodules 
 $d_i:A\otimes_Ek\A_i\otimes_EA\to A\otimes_Ek\A_{i-1}\otimes_EA$ for 
 $1\leq i\leq N$. If
 \begin{enumerate}
  \item $d_{i-1}\circ d_i=0$ for all $i$, $1\leq i\leq N$,
  \item $(d_i-\delta_i)(1\otimes q\otimes 1)\in 
   \llrr{\overline{\LL}_{i-1}^\prec(q)}_k$ for all $i\in\{1,\dots,N\}$ and for
    all $q\in\A_i$,
 \end{enumerate}
 then the complex
 \[
  \xymatrix{A\otimes_Ek\A_N\otimes_EA\ar[r]^-{d_N}&\cdots\ar[r]^-{d_1}&
            A\otimes_Ek\A_0\otimes_EA\ar[r]^-{d_0}&
            A\otimes_EA\ar[r]^-{d_{-1}}&
            A\ar[r]&0}
 \]
 is exact.
\end{theorem}

\begin{theorem}
\label{teo2}
 There exist $A$-bimodule morphisms 
 $d_i:A\otimes_Ek\A_i\otimes_EA\to A\otimes_Ek\A_{i-1}\otimes_EA$ for
 $i\in\NN_0$ and $d_{-1}:A\otimes_EA\to A$ such that 
 \begin{enumerate}
  \item $d_{i-1}\circ d_i=0$, for all $i\in\NN_0$,
  \item $(d_i-\delta_i)(1\otimes q\otimes1)\in
   \llrr{\overline{\LL}_{i-1}^\prec(q)}_\ZZ$ for all $i\geq-1$ and $q\in\A_i$.
 \end{enumerate}
\end{theorem}
 We will carry out the proofs of these theorems in the following section.


\section{Proofs of the theorems}
\label{s:Proofs}

We keep the same notations and conditions of the previous section.
We start by proving some technical lemmas.

\begin{lemma}
\label{lemma1}
 Given $n\geq0$, the following equalities hold
 \begin{enumerate}
  \item $\delta_n\circ\pi_n = \pi_{n-1}\circ f_n$,
  \item $\delta_n'\circ\pi_n' = \pi_{n-1}'\circ f_n$.
 \end{enumerate}
\end{lemma}
The proof is straightforward after the definitions.

\medskip


Next we prove three lemmas where we study how various maps defined in Section \ref{s:Resolution} behave with respect to the order. 

\begin{lemma}
\label{lemma2}
 For all $n\in\NN_0$ and $\mu q\in k^\times Q_{\geq0}$, the  images by 
 $\pi_n$ of $\LL_n^\preceq(\mu q)$ and of 
 $\LL_n^\prec(\mu q)$ are respectively  contained in 
 $\llrr{\overline{\LL}_n^\preceq(\mu q)}_\ZZ$ and in 
 $\llrr{\overline{\LL}_n^\prec(\mu q)}_\ZZ$.
\end{lemma}

\begin{proof}
 Given $n\in\NN_0$, $\mu q\in k^\times Q_{\geq0}$ and 
 $x=\lambda a\otimes p\otimes c\in\LL_n^\preceq(\mu q)$, where 
 $a,c\in Q_{\geq0}$ and $p\in\A_n$, suppose $\beta(a)=\sum_i\lambda_ib_i$ and
 $\beta(c)=\sum_j\lambda_j'b_j'$. Since $\beta(a)\preceq a$ and 
 $\beta(c)\preceq c$, then $\lambda_ib_i\preceq a$ and 
 $\lambda_j'b_j'\preceq c$ for all $i,j$. This implies
 \[
  \lambda\lambda_i\lambda_jb_ipb_j'\preceq \lambda apc\preceq \mu q
 \]
 and so $\lambda\lambda_i\lambda_j'\pi(b_i)\otimes p\otimes\pi(b_j')$ belong
 to $\overline{\LL}_n^\preceq(\mu q)$ for all $i,j$. The result follows from the equalities 
 \[\pi_n(x)=\lambda\pi(a)\otimes p\otimes \pi(c) = 
 \lambda\pi(\beta(a))\otimes p\otimes\pi(\beta(c))=
 \sum_{i,j}\lambda\lambda_i\lambda_j'\pi(b_i)\otimes p\otimes\pi(b_j').
 \]

 The proof of the second part is analogous.
 \end{proof}

 \begin{corollary}
 \label{coro1}
 Let $n\geq-1$ and $\mu q\in k^\times Q_{\geq0}$. Keeping the same notations of
 the proof of the previous lemma, we conclude that
 \begin{enumerate}[i)]
  \item if $x\in\overline{\LL}_n^\preceq(\mu q)$, then $\lambda\pi(a)x\pi(c)\in
           \llrr{\overline{\LL}_n^\preceq(\lambda\mu aqc)}_\ZZ$,
  \item if $x\in\overline{\LL}_n^\prec(\mu q)$, then $\lambda\pi(a)x\pi(c)\in
           \llrr{\overline{\LL}_n^\prec(\lambda\mu aqc)}_\ZZ$.
 \end{enumerate}
\end{corollary}

\medskip

\begin{lemma}
\label{lemma4}
 Given $n\in\NN_0$ and $\mu q\in k^\times Q_{\geq0}$, there are inclusions
\begin{multicols}{2}
 \begin{enumerate}[i)]
  \item $\delta_n(\overline{\LL}_n^\preceq(\mu q))\subseteq\llrr{\overline{\LL}_{n-1}^\preceq(\mu q)}_\ZZ$,
  \item $\delta_n(\overline{\LL}_n^\prec(\mu q))\subseteq\llrr{\overline{\LL}_{n-1}^\prec(\mu q)}_\ZZ$,
  \item $s_n(\overline{\LL}_{n-1}^\preceq(\mu q))\subseteq\llrr{\overline{\LL}_{n}^\preceq(\mu q)}_\ZZ$,
  \item $s_n(\overline{\LL}_{n-1}^\prec(\mu q))\subseteq\llrr{\overline{\LL}_{n}^\prec(\mu q)}_\ZZ$.
 \end{enumerate}
\end{multicols}
\end{lemma}

\begin{proof}
From $x=\lambda\pi(b)\otimes p\otimes\pi(b')\in\overline{\LL}_n^\preceq(\mu q)$, with
$b,b'\in\B$ and $p\in\A_n$, we get $i_n(x) = \lambda b\otimes p\otimes b'$. This element belongs to $\LL_n^\preceq(\mu q)$ and this implies that $f_n(\lambda b\otimes p\otimes b')$ 
belongs to $\llrr{\LL_{n-1}^\preceq(\mu q)}_\ZZ$, by Remark \ref{remark1}. As a consequence of Lemma \ref{lemma2} 
we obtain that $\delta_n(x) = \pi_{n-1}(f_n(\lambda b\otimes p\otimes b'))$ 
belongs to $\llrr{\overline{\LL}_{n-1}^\preceq(\mu q)}_\ZZ$. The proofs of the other statements are similar.
\end{proof}

\begin{lemma}
\label{lemma5}
 Given $n\geq -1$ and $\mu q\in k^\times Q_{\geq0}$, if 
 $x=\lambda a\otimes p\otimes c\in\LL_n^\preceq(\mu q)$ is such that $\pi_n'(x)=0$, then
 \[
  \pi_n(x)\in\llrr{\overline{\LL}_n^\prec(\mu q)}_\ZZ.
 \]
\end{lemma}

\begin{proof}
 By hypothesis we get that $0=\pi_n'(x)=\pi'(a)\otimes p\otimes\pi'(c)$. The only 
 possibilities are $\pi'(a)=0$ or $\pi'(c)=0$, this is, $a\notin\B$ or 
 $c\notin\B$, namely $\beta(a)\prec a$ or $\beta(c)\prec c$.
 
 Writing $\beta(a)=\sum_i\lambda_ib_i$ and 
$\beta(c)=\sum_j\lambda_j'b_j'$, 
we deduce that $\lambda\lambda_i\lambda_j'b_ipb_j\prec \mu q$ for
 all $i,j$. As a consequence, 
 $\sum_{i,j}\lambda\lambda_i\lambda_j'\pi(b_i)\otimes p\otimes\pi(b_j')\in
 \llrr{\overline{\LL}_n^\prec(\mu q)}_\ZZ$.
 
 The proof ends by computing
 \[
  \pi_n(x)=\pi_n(\beta(x)) = 
  \pi_n(\sum_{i,j}\lambda\lambda_i\lambda_j'b_i\otimes p\otimes b_j')=
  \sum_{i,j}\lambda\lambda_i\lambda_j'\pi(b_i)\otimes p\otimes \pi(b_j').
 \]
\end{proof}

The importance of the preceding lemmas is that they guarantee
how differentials and morphisms used for the homotopy behave with respect to
the order. This is stated explicitly in the following corollary.

\begin{corollary}
\label{coro:diff_y_ord}
 Given $n \ge 1$, $\mu q\in k^\times Q_{\geq0}$ and $x\in\overline{\LL}_n^\preceq(\mu q)$,
 the following facts hold:
 \begin{enumerate}
  \item $\delta_{n-1}\circ\delta_n(x)\in\llrr{\overline{\LL}_{n-2}^\prec(\mu q)}_\ZZ$,
  \label{coro:diff_y_ord1}
  \item $x-\delta_{n+1}\circ s_{n+1}(x)-s_n\circ\delta_n(x)\in\llrr{\overline{\LL}_n^\prec(\mu q)}_\ZZ$.
  \label{coro:diff_y_ord2}
  \end{enumerate}
\end{corollary}

\begin{proof}
 Let us first write $x=\lambda\pi(b)\otimes p\otimes\pi(b')$
 with $b,b'\in\B$ and $x'\colonequals i_n(x) = \lambda b\otimes p\otimes b'$.
 Lemma \ref{lemma1} implies that
 \[
  \delta_{n-1}\circ\delta_n(x) = \delta_{n-1}\circ\delta_n\circ\pi_n(x')
    =\delta_{n-1}\circ\pi_{n-1}\circ f_n(x') 
    = \pi_{n-2}\circ f_{n-1}\circ f_n(x').
 \]
 By Remark \ref{remark1}, $f_{n-1}\circ f_n(x')\in\LL_{n-2}^\preceq(\mu q)$. Next, by 
 Lemma \ref{lemma5}, in order to prove that $\delta_{n-1}\circ\delta_n(x)\in
 \llrr{\overline\LL_{n-2}^\prec(\mu q)}_\ZZ$, it suffices to verify that 
 $\pi_{n-2}'\circ f_{n-1}\circ f_n(x')=0$, which is in fact true using Lemma 
 \ref{lemma1}, and the fact that 
 $(A_S\otimes_Ek\A_\bullet\otimes_EA_S,\delta_\bullet')$ is exact.
 
 In order to prove (\ref{coro:diff_y_ord2}), we first remark that if 
 $k\in\NN_0$ and $y\in\llrr{\LL_k^\preceq(\mu q)}_\ZZ$, then $i_k'\circ\pi_k'(y)-i_k\circ\pi_k(y)
 \in\llrr{\LL_k^\prec(\mu q)}_\ZZ$. Indeed, let us write $y=\lambda a\otimes p\otimes c\in
 \LL_k^\preceq(\mu q)$. In case $a\in\B$ and $c\in\B$, there are equalities $i_k'\circ\pi_k'(y)=
 y=i_k\circ\pi_k(y)$, and so the difference is zero. If either $a\notin\B$ or
 $c\notin\B$, then $\pi_k'(y)=0$ and in this case Lemma \ref{lemma5} implies
 that $\pi_k(y)\in\llrr{\overline\LL_k^\prec(\mu q)}_\ZZ$. So, 
 $i_k\circ\pi_k(y)\in\llrr{\LL_k^\prec(\mu q)}_\ZZ$ and the difference we are 
 considering  belongs to $\llrr{\LL_k^\prec(\mu q)}_\ZZ$.
 
 Fix now $x=\lambda\pi(b)\otimes p\otimes\pi(b')$ and 
 $x'=i_n(x)=\lambda b\otimes p\otimes b'$, with $b,b'\in\B$.
 
 Since $x'=i_n'\circ\pi_n'(x')$,
 \[
 \begin{aligned}
  x-\delta_{n+1}\circ s_{n+1}(x)-s_n\circ\delta_n(x)
     =\pi_n(x')&-\pi_n(f_{n+1}\circ i_{n+1}\circ\pi_{n+1}\circ S_{n+1}(x'))\\
     &-\pi_n(S_n\circ i_{n-1}\circ\pi_{n-1}\circ f_n(x')).
 \end{aligned}
 \]
 The previous comments and Remark \ref{remark1} allow us to write that
 \[
  \begin{aligned}
   &\pi_n\circ f_{n+1}\circ(i_{n+1}'\circ\pi_{n+1}'
            -i_{n+1}\circ\pi_{n+1})\circ S_{n+1}(x')
            \in\langle\overline{\LL}_n^\prec(\mu q)\rangle_\ZZ,\\
   &\pi_n\circ S_n\circ(i_{n-1}'\circ\pi_{n-1}' 
            - i_{n-1}\circ\pi_{n-1})\circ f_n(x')
            \in\langle\overline{\LL}_n^\prec(\mu q)\rangle_\ZZ.
  \end{aligned}
 \]
 It is then enough to prove that
 \[
  \pi_n(x'-f_{n+1}\circ i_{n+1}'\circ\pi_{n+1}'\circ S_{n+1}(x')
  - S_n\circ i_{n-1}'\circ\pi_{n-1}'\circ f_n(x'))\in
  \langle\overline{\LL}_n^\prec(\mu q)\rangle_\ZZ,
 \]
 but
 \[
  \begin{aligned}
   &\pi_n'(x'-f_{n+1}\circ i_{n+1}'\circ\pi_{n+1}'\circ S_{n+1}(x')
      - S_n\circ i_{n-1}'\circ\pi_{n-1}'\circ f_n(x'))\\
  &=\pi_n'(x')-\delta_{n+1}'\circ s_{n+1}'(\pi_n'(x')) 
      - s_n'\circ\delta_n'(\pi_n'(x'))\\
  &=0.
  \end{aligned}
 \]
 Finally, we deduce from Lemma \ref{lemma5} that
 \[
  \pi_n(x'-f_{n+1}\circ i_{n+1}'\circ\pi_{n+1}'\circ S_{n+1}(x')
     -S_n\circ i_{n-1}'\circ\pi_{n-1}'\circ f_n(x'))
     \in\langle\overline{\LL}_n^\prec(\mu q)\rangle_\ZZ.
 \]
\end{proof}

Next we prove another technical lemma that shows how to control the 
differentials.

\begin{lemma}
\label{lemma7}
 Fix $n\in\NN_0$, let $R$ be either $k$ or $\ZZ$.
\begin{enumerate}
 \item If $d:A\otimes_Ek\A_n\otimes_EA\to A\otimes_Ek\A_{n-1}\otimes_EA$ is a
 morphism of $A$-bimodules such that $(d-\delta_n)(1\otimes p\otimes 1)\in
 \langle\overline\LL_{n-1}^\prec(p)\rangle_R$ for all $p\in\A_n$, then given 
 $x\in\llrr{\overline\LL_n^\preceq(\mu q)}_R$, $(d-\delta_n)(x)
 \in\langle\overline{\LL}_{n-1}^\prec(\mu q)\rangle_R$ for all 
 $\mu q\in k^\times Q_{\geq0}$.
 
 \item If $\rho: A\otimes_Ek\A_n\otimes_EA\to A\otimes_Ek\A_{n+1}\otimes_EA$ 
 is a morphism of $A-E$-bimodules such that 
 $(\rho-s_n)(1\otimes p\otimes\pi(b))\in\llrr{\overline\LL_{n+1}^\prec(pb)}_R$,
 for all $p\in\A_n$ and $b\in\B$, then for all $x\in\langle\overline{\LL}_n^\preceq(\mu q)\rangle_R$,
 $(\rho-s_n)(x)$ belongs to $\langle\overline{\LL}_{n+1}^\prec(\mu q)\rangle_R$ for all
 $\mu q\in k^\times Q_{\geq0}$.
\end{enumerate}
\end{lemma}

\begin{proof}
 Given $\mu q\in k^\times Q_{\geq0}$ and $x\in
 \llrr{\overline\LL_n^\preceq(\mu q)}_R$, let us see that $(d-\delta_n)(x)\in
 \langle\overline{\LL}_{n-1}^\prec(\mu q)\rangle_R$. It suffices to prove the statement for 
 $x=\lambda\pi(b)\otimes p\otimes\pi(b')\in \overline{\LL}_n^\preceq(\mu q)$.
 
 By hypothesis, $(d-\delta_n)(1\otimes p\otimes1)$ belongs to 
 $\langle\overline{\LL}_{n-1}^\prec(p)\rangle_R$, so $(d-\delta_n)(x)$ equals
 $\lambda\pi(b)(d-\delta_n)(1\otimes p\otimes 1)\pi(b')$ and it belongs to
 $\llrr{\overline\LL_{n-1}^\prec(\lambda bpb')}_R\subseteq
  \langle\overline{\LL}_{n-1}^\prec(\mu q)\rangle_R$, using Corollary \ref{coro1}.
 
 The second part is analogous.
\end{proof}
 
 Next proposition will provide the remaining necessary tools for the proofs
 of Theorem \ref{teo1} and Theorem \ref{teo2}.

\begin{proposition}
\label{propfinal}
Fix $n\in\NN_0$. Suppose that for each $i\in\{0,\dots,n\}$ there are 
morphisms of $A$-bimodules 
$d_i:A\otimes_Ek\A_i\otimes_EA\to A\otimes_Ek\A_{i-1}\otimes_EA$, and 
morphisms 
of $A-E$-bimodules $\rho_i:A\otimes_Ek\A_{i-1}\otimes_EA\to 
A\otimes_Ek\A_i\otimes_EA$. Denote $d_{-1}=\mu$ and define 
$\rho_{-1}:A\to A\otimes_EA$ as $\rho(a)=a\otimes1$.
 
If the following conditions hold,
 \begin{enumerate}[(i)]
  \item $d_{i-1}\circ d_i=0$ for all $i\in\{0,\dots,n\}$,
  \item $(d_i-\delta_i)(1\otimes q\otimes 1)\in\llrr{\overline\LL_{i-1}^\prec(q)}_R$
   for all $i\in\{0,\dots,n\}$ and for all $q\in\A_i$,
  \item for all $i\in\{-1,\dots,n-1\}$ and for all $x\in 
A\otimes_Ek\A_i\otimes_EA$, $x=d_{i+1}\circ\rho_{i+1}(x) + \rho_i\circ 
d_i(x)$,
  \item $(\rho_i-s_i)(1\otimes 
q\otimes\pi(b))\in\llrr{\overline\LL_{i}^\prec(qb)}_R$
   for all $i\in\{0,\dots,n\}$, for all $q\in\A_i$ and for all $b\in\B$,
 \end{enumerate}
then:
\begin{enumerate}
 \item
 \label{propfinal2} If $d_{n+1}:A\otimes_Ek\A_{n+1}\otimes_EA\to 
A\otimes_Ek\A_n\otimes_EA$ is a map satisfying the following conditions:
 \begin{enumerate}[(i)]
   \item $d_n\circ d_{n+1}=0$,
   \item $(d_{n+1}-\delta_{n+1})(1\otimes q\otimes 1)
     \in\llrr{\overline\LL_n^\prec(q)}_R$,
 \end{enumerate}
then there exists a morphism $\rho_{n+1}:A\otimes_Ek\A_n\otimes_EA\to 
A\otimes_Ek\A_{n+1}\otimes_EA$ of $A-E$ bimodules such
that 
  \begin{enumerate}
   \item for all 
  $x\in A\otimes_Ek\A_n\otimes_EA$, $x=d_{n+1}\circ s_{n+1}(x) + s_n\circ d_n(x)$
   \item for all $q\in \mathcal{A}_n$ and for all $b \in \B$,  ($\rho_{n+1}-s_{n+1})(1\otimes q \otimes \pi(b)) \in \llrr{\LL_{n+1}^\prec(qb)}_R$.
 \end{enumerate}
 \item
 \label{propfinal1} there exists a morphism of $A$-bimodules 
 $d_{n+1}:A\otimes_Ek\A_{n+1}\otimes_EA\to A\otimes_Ek\A_n\otimes_EA$ such
 that
 \begin{enumerate}[(i)]
  \item $d_n\circ d_{n+1}=0$,
   \item $(d_{n+1}-\delta_{n+1})(1\otimes q\otimes 1)
     \in\llrr{\overline\LL_n^\prec(q)}_R$.
 \end{enumerate} 
 
\end{enumerate}
\end{proposition}

\begin{proof}
 In order to prove (\ref{propfinal1}), fix $q\in\A_{n+1}$. By Lemma 
 \ref{lemma4}, $\delta_{n+1}(1\otimes q\otimes 1)$ belongs to 
 $\llrr{\overline\LL_n^\preceq(q)}_\ZZ$ and using Lemma \ref{lemma7}, 
 $(d_n-\delta_n)(\delta_{n+1}(1\otimes q\otimes 1))$ belongs to 
 $\llrr{\overline\LL_{n-1}^\prec(q)}_R$. Corollary \ref{coro:diff_y_ord} tells
 us that $\delta_n\circ\delta_{n+1}(1\otimes q\otimes 1)$ is in
 $\llrr{\overline\LL_{n-1}^\prec(q)}_\ZZ$. We deduce from the equality
 \[
 d_n(\delta_{n+1}(1\otimes q\otimes 1))=
     \delta_n\circ\delta_{n+1}(1\otimes q\otimes 1) 
     + (d_n-\delta_n)(\delta_{n+1}(1\otimes q\otimes 1))
 \]
 that $d_n(\delta_{n+1}(1\otimes q\otimes 1))$ belongs to 
 $\llrr{\overline\LL_{n-1}^\prec(q)}_R$.

 Let us define $\tilde d_{n+1}: A\times k\A_{n+1}\times A\to 
 A\otimes_Ek\A_n\otimes_EA$ by
 \[
  \tilde d_{n+1}(a,q,c)=a\delta_{n+1}(1\otimes q\otimes 1)c 
   -a\rho_n(d_n(\delta_{n+1}(1\otimes q\otimes 1)))c,
 \]
 for $a,c\in A$, $q\in\A_{n+1}$. The map $\tilde d_{n+1}$ is $E$-multilinear
 and balanced, and it induces a unique map
 \[
  d_{n+1}:A\otimes_Ek\A_{n+1}\otimes_EA\to A\otimes_Ek\A_n\otimes_EA.
 \]
 It is easy to verify that $d_{n+1}$ is in fact a morphism of $A$-bimodules.
 
 Putting together the equality $\rho_n=s_n+(\rho_n-s_n)$ and Lemmas \ref{lemma4}
 and \ref{lemma7}, we obtain that $(d_{n+1}-\delta_{n+1})(1\otimes q\otimes 1)= -\rho_n\circ d_n\circ \delta_{n+1}(1\otimes q\otimes 1)$
 belongs to $\llrr{\overline\LL_n^\prec(q)}_R$.
 Moreover, given $x\in A\otimes_Ek\A_{n-1}\otimes_EA$, 
 $x=d_n\circ\rho_n(x)+\rho_{n-1}\circ d_{n-1}(x)$, choosing $x=d_n(\delta_{n+1}(1\otimes q\otimes 1))$
 yields the equality
 \[
  d_n\circ \delta_{n+1}(1\otimes q\otimes 1) = d_n\circ\rho_n\circ d_n\circ\delta_{n+1}(1\otimes q\otimes 1)
 \]
 which proves that $d_n\circ d_{n+1}=0$.
 
 For the proof of (\ref{propfinal2}), fix $q\in\A_n$ and $b\in\B$. Using 
Lemmas
 \ref{lemma4} and \ref{lemma7}, we deduce that the element
 \[
 \begin{aligned}
  &1\otimes q\otimes\pi(b) -\rho_n\circ d_n(1\otimes q\otimes\pi(b)) \\
  &= 1\otimes q\otimes\pi(b) - \rho_n\circ\delta_n(1\otimes q\otimes\pi(b))-\rho_n\circ(d_n-\delta_n)(1\otimes q\otimes\pi(b))
 \end{aligned}
 \]
 differs from 
 $1\otimes q\otimes\pi(b)-\rho_n\circ\delta_n(1\otimes q\otimes\pi(b))$ by 
 elements in $\llrr{\overline\LL_n^\prec(qb)}_R$. We will write that
 \[
  (\id-\rho_n\circ\delta_n+\rho_n\circ(d_n-\delta_n))(1\otimes q\otimes\pi(b))
  \equiv \id-\rho_n\circ\delta_n(1\otimes q\otimes\pi(b)) 
  \:\mbox{mod}\llrr{\overline\LL_n^\prec(qb)}_R.
 \]
 Also,
 \[
  \begin{aligned}
  (\id-\rho_n\circ\delta_n)(1\otimes q\otimes\pi(b))&\equiv
  (\id-s_n\circ\delta_n)(1\otimes q\otimes\pi(b))\:
    \mbox{mod}\llrr{\overline\LL_n^\prec(qb)}_R\\
  &\equiv \delta_{n+1}\circ s_{n+1}(1\otimes q\otimes\pi(b))\:
    \mbox{mod}\llrr{\overline\LL_n^\prec(qb)}_R\\
  &\equiv d_{n+1}\circ s_{n+1}(1\otimes q\otimes\pi(b))\:
    \mbox{mod}\llrr{\overline\LL_n^\prec(qb)}_R.
  \end{aligned}
 \]
 We deduce from this that there exists a unique $\xi\in\llrr{\overline\LL_n^\prec(qb)}_R$
 such that 
 \[(\id-\rho_n\circ d_n)(1\otimes q\otimes\pi(b)) = 
  d_{n+1}\circ s_{n+1}(1\otimes q\otimes\pi(b)) + \xi.
 \]
 It is evident that
 $\xi$ belongs to the kernel of $d_n$.
 
 The order $\preceq$ satisfies the descending chain condition, so we can use
 induction
 on
 $(k^\times Q_{\geq0},\preceq)$. If there is no $\lambda p\in k^\times Q_{\geq0}$ is 
 such that $\lambda p\prec qb$, then $\xi=0$ and we define 
 $\rho_{n+1}(1\otimes q\otimes\pi(b))=s_{n+1}(1\otimes q\otimes\pi(b))$. 
 Inductively, suppose that $\rho_{n+1}(\xi)$ is defined. The equality 
 $d_n(\xi)=0$ implies that $\xi=d_{n+1}\circ\rho_{n+1}(\xi)$ and 
 \[
  (\id-\rho_n\circ d_n)(1\otimes q\otimes\pi(b)) =
    d_{n+1}(s_{n+1}(1\otimes q\otimes\pi(b)) + \rho_{n+1}(\xi)).
 \]
 We define $\rho_{n+1}(1\otimes q\otimes\pi(b))\colonequals 
  s_{n+1}(1\otimes q\otimes\pi(b)) + \rho_{n+1}(\xi)$.
  
 Lemmas \ref{lemma4} and \ref{lemma7} assure that $\rho_{n+1}(\xi)$ belongs to
 $\llrr{\overline\LL_{n+1}^\prec(qb)}_R$, and as a consequence
 \[
  \rho_{n+1}(1\otimes q\otimes\pi(b))-s_{n+1}(1\otimes q\otimes\pi(b))\in
  \llrr{\overline\LL_{n+1}^\prec(qb)}_R.
 \]
 
\end{proof}
 We are now ready to prove the theorems.
\begin{proof}[Proof of Theorem \ref{teo1}] 
  We will prove the existence of an $A-E$-bimodule map $\rho_0:
 A\otimes_Ek\A_{-1}\otimes_EA\to A\otimes_Ek\A_0\otimes_EA$ satisfying 
 $d_0\circ\rho_0 + \rho_{-1}\circ d_{-1} = \id$, where $d_{-1}=\mu$ and $\rho_{-1}(a)=s_{-1}(a)=a\otimes 1$ for all $a \in A$. 
 Once this achieved, we apply
 Proposition \ref{propfinal} inductively with $R=k$, for all $n$ such that
 $0\leq n\leq N-1$, obtaining this way an homotopy retraction of the complex
 \[
  \xymatrix{A\otimes_Ek\A_N\otimes_EA\ar[r]^-{d_N}&
            \cdots\ar[r]^-{d_0}&
            A\otimes_EA\ar[r]^-{d_{-1}}&
            A\ar[r]&
            0}
 \]
 proving thus that it is exact. 


 Given $b=b_k\cdots b_1\in\B$, with $b_i\in Q_1$, $1\leq i\leq k$,
 \[
  s_0(1\otimes\pi(b))=
   -\sum_i\pi(b_k\cdots b_{k-i+1})\otimes b_{k-i}\otimes\pi(b_{k-i-1}\cdots b_1).
 \]
 On one hand 
 $1\otimes\pi(b)-\pi(b)\otimes1=1\otimes\pi(b)-s_{-1}(d_{-1}(1\otimes\pi(b)))$
 and on the other hand the left hand term equals 
 $\delta_0(s_0(1\otimes\pi(b)))$, yielding
 $1\otimes\pi(b)-s_{-1}(1\otimes\pi(b))=\delta_0(s_0(1\otimes\pi(b))$. By 
 hypothesis, $(d_0-\delta_0)(1\otimes\pi(b))$ belongs to 
 $\llrr{\overline\LL_{-1}^\prec(b)}_k$, and  so there exists $\xi\in
 \llrr{\overline\LL_{-1}^\prec(b)}_k$ such that 
 \[
  1\otimes\pi(b)-s_{-1}(d_{-1}(1\otimes\pi(b)))=d_0(s_0(1\otimes\pi(b)))+\xi.
 \]
 It follows that $d_{-1}(\xi)=0$. Suppose first that there exists no 
 $\lambda p\in k^\times Q_{\geq0}$ such that $\lambda p\prec b$.
 
 In this case $\xi=0$ and we define $\rho_0(1\otimes\pi(b)) 
 = s_0(1\otimes\pi(b))$. Inductively, suppose that $\rho_0(\xi)$ is defined for any $\xi$ such that $d_{-1}(\xi)=0$.
 Since in this case $\xi=d_0(\rho_0(\xi))$, we set
 $\rho_0(1\otimes\pi(b)) \colonequals s_0(1\otimes\pi(b)) + \rho_0(\xi)$.
\end{proof}
  
\begin{proof}[Proof of Theorem \ref{teo2}]
 It follows from the proof of Theorem \ref{teo1} that 
 \[
  1\otimes\pi(b)=(s_{-1}\circ d_{-1} + \delta_0\circ s_0)(1\otimes\pi(b))
 \]
 and so $s_{-1}\circ d_{-1} + \delta_0\circ s_0 = \id_{A\otimes_EA}$. Setting
 $d_0\colonequals\delta_0$, the theorem follows applying Proposition 
 \ref{propfinal} for $R=\ZZ$.
\end{proof}

We finish this section showing that this construction is a generalization of
Bardzell's resolution for monomial algebras.

\begin{proposition}
 Given an  algebra $A$, 
 let $(A\otimes_Ek\A_\bullet\otimes_EA,d_\bullet)$ be a 
 resolution of $A$ as $A$-bimodule such that $d_\bullet$ satisfies the 
 hypotheses of Theorem \ref{teo1}. If $p\in\A_n$ is such that $r(p)=0$ or
 $r(p)=p$ for every reduction $r$, then for all $a,c\in kQ$,
 \[
  d_n(\pi(a)\otimes p\otimes\pi(c)) = \delta_n(\pi(a)\otimes p\otimes\pi(c)).
 \]
\end{proposition}

\begin{proof}
 By hypothesis, there exists no $\lambda'p'\in k^\times Q_{\geq0}$ such that
 $\lambda'p'\prec p$, so $\LL_{n-1}^\prec(p)=\{0\}$ and 
 $d_n(1\otimes p\otimes 1)=\delta_n(1\otimes p\otimes 1)$. Given 
 $a,c\in kQ$ we deduce from the previous equality that 
 \[
  d_n(\pi(a)\otimes p\otimes\pi(c)) -\delta_n(\pi(a)\otimes p\otimes\pi(c))
  = \pi(a)(d_n(1\otimes p\otimes 1) - \delta_n(1\otimes p\otimes 1))\pi(c)=0.
 \]
\end{proof}

\begin{corollary}
 Suppose 
 the algebra $A=kQ/I$ has a monomial 
 presentation. Choose a reduction system $\mathcal R$ whose pairs have the monomial relations generating the ideal $I$ as first coordinate and $0$ as second coordinate. 
 In this case, the only maps $d$ verifying the hypotheses of Theorem
 $4.2$ are those of Bardzell's resolution.
\end{corollary}

\section{Morphisms in low degrees}
\label{s:Morphisms}
 In this section we describe the morphisms appearing in lower degrees of the resolution.
 
 Let us consider the following data: an algebra $A=kQ/I$ and a reduction system $\mathcal R$ satisfying condition $\lozenge$. 
 
 We start by recalling the definition of $\delta_0$ and $\delta_{-1}$. For $a,c\in kQ$, 
 $\alpha\in Q_1$,
 \[
  \begin{aligned}
   &\delta_{-1}:A\otimes_EA\to A,& &\delta_{-1}(\pi(a)\otimes\pi(c))=\pi(ac) \mbox{ and }\\
   &\delta_0:A\otimes_Ek\A_0\otimes_EA\to A\otimes_EA,& 
     &\delta_0(\pi(a)\otimes \alpha\otimes\otimes(c)) = \pi(a\alpha)\otimes\pi(c)
                                               - \pi(a)\otimes\pi(\alpha c).
  \end{aligned}
 \]

\begin{definition} We state some definitions.
\begin{itemize}
 \item Let $\phi_0:kQ\to A\otimes_Ek\A_0\otimes_EA$ be the unique $k$-linear map such that
 \[
  \phi_0(c)=
  \sum_{i=1}^n\pi(c_n\cdots c_{i+1})\otimes c_i\otimes\pi(c_{i-1}\cdots c_1)
 \]
 for $c\in Q_{\geq0}$, $c=c_n\cdots c_1$ with $c_i\in Q_1$ for all $i$, 
 $1\leq i\leq n$.
 
 \item Given a basic reduction $r=r_{a,s,c}$, let 
 $\phi_1(r,-):kQ\to A\otimes_Ek\A_1\otimes_EA$ be the unique $k$-linear map 
 such that, given $p\in Q_{\ge 0}$ 
 \[
  \phi_1(r,p) = \begin{dcases*}
                 \pi(a)\otimes s\otimes \pi(c), & if $p=asc$,\\
                 0                              & if not.
                \end{dcases*}
 \]
 In case $r=(r_n,\dots,r_1)$ is a reduction, where $r_i$ is a basic reduction
 for all $i$, $1\leq i\leq n$, we denote $r'=(r_n,\dots,r_2)$ and we define in a 
 recursive way the map $\phi_1(r,-)$ as the unique $k$-linear map from 
$kQ$ to
 $A\otimes_Ek\A_1\otimes_EA$ such that
 \[
  \phi_1(r,p)=\phi_1(r_1,p) + \phi_1(r',r_1(p)).
 \]
 \item Finally, we define an $A$-bimodule morphism 
 $d_1:A\otimes_Ek\A_1\otimes_EA\to A\otimes_Ek\A_0\otimes_EA$ by the 
 equality
 \[
  d_1(1\otimes s\otimes 1)=\phi_0(s)-\phi_0(\beta(s)),\mbox{ for all }s\in\A_1.
 \]
\end{itemize}
\end{definition}
 Next we prove four lemmas necessary to the description of 
 the complex in low degrees.
 
\begin{lemma}
\label{lemma:phi_1menores}
 Let us consider $p\in Q_{\geq0}$ and $x\in kQ$ such that $x\prec p$. For any
 reduction $r$ the element $\phi_1(r,x)$ belongs to 
 $\llrr{\overline\LL_1^\prec(p)}_\ZZ$.
\end{lemma}

\begin{proof}
 We will first prove the result for $x=\mu q\in k^\times Q_{\geq0}$. The general
 case will then follow by linearity. Fix $x=\mu q\in k^\times Q_{\geq0}$. We
 will use an inductive argument on ${(k^\times Q_{\geq0},\preceq)}$. 
 
 To start the induction, suppose first that there exists no 
 $\mu'q'\in k^\times Q_{\geq0}$ and that $\mu'q'\prec\mu q=x$. In this case,
 every basic reduction $r_{a,s,c}$ satisfies either $r_{a,s,c}(x)=x$ or 
 $r_{a,s,c}=0$. In the first case, $asc\neq q$ and so $\phi_1(r_{a,s,c},x)=0$.
 In the second case, $asc= q$, so 
 $\phi_1(r_{a,s,c},x)=\mu\pi(a)\otimes s\otimes\pi(c)$.
 
 Given an arbitrary reduction $r=(r_n,\dots,r_1)$ with $r_i$ basic for all $i$,
 there are three possible cases.
 
 \begin{enumerate}
  \item $r_1(x)=x$ and $n>1$,
  \label{lemmaitemdem1}
  \item $r_1(x)=x$ and $n=1$,
  \label{lemmaitemdem2}
  \item $r_1(x)=0$.
  \label{lemmaitemdem3}
 \end{enumerate}
 Denote $r'=(r_n,\dots,r_2)$ as before and $r_1=r_{a,s,c}$. In case \ref{lemmaitemdem1}), 
 $\phi_1(r,x)=\phi_1(r',x)$. In case \ref{lemmaitemdem3}), 
 $\phi_1(r,x)=\phi_1(r_1,x)=0.$ Finally, in case \ref{lemmaitemdem2}), 
 $\phi_1(r,x)=\phi_1(r_1,x)=\mu\pi(a)\otimes s\otimes\pi(c)$. Using Lemma 
 \ref{lemma2}, we obtain that in all three cases 
 $\phi_1(r,x)\in\llrr{\overline\LL_1^\prec(p)}_\ZZ$.
 
 Next, suppose that $x=\mu q$ and that the result holds for 
 $\mu'q'\in k^\times Q_{\geq0}$ such that $\mu'q'\prec\mu q=x$. Let us consider
 $r,r_1$ and $r'$ as before. Again, there are three possible cases:
 \begin{enumerate}
  \item $asc=q$,
  \label{lemmaitemdem11}
  \item $asc\neq q$ and $n>1$,
  \label{lemmaitemdem22}
  \item $asc\neq q$ and $n=1$.
  \label{lemmaitemdem33}
 \end{enumerate}
 Case \ref{lemmaitemdem33}) is immediate, since in this situation $\phi_1(r,x)=
 0$. The second case reduces to the other ones, since $\phi_1(r,x)=\phi_1(r',x)$
 In the
 first case,
 \[
  \phi_1(r,x)=\mu\pi(a)\otimes s\otimes\pi(c) + \phi_1(r',r_1(x)).
 \]
 We know that $r_1(x)\prec x$, and we may write it as a finite sum $r_1(x)=\sum_i\mu_iq_i$.
 Using the inductive hypothesis, we deduce that
 $\phi_1(r,x)\in\llrr{\overline\LL_1^\prec(p)}_\ZZ$.
\end{proof}

\begin{lemma}
\label{lemma:delta0d1igual0}
 For all $x\in A\otimes_Ek\A_1\otimes_EA$, $x$ belongs to the kernel of $\delta_0\circ d_1(x)$.
\end{lemma}

\begin{proof}
 Since these maps are morphisms of $A$-bimodules, we may suppose 
 $x=1\otimes s\otimes 1$, with $s\in \mathcal{A}_1$. A direct computation gives
 \[
  \delta_0(d_1(1\otimes s\otimes 1)=\delta_0(\phi_0(s)-\phi_0(\beta(s))) =
    \pi(s)\otimes1-1\otimes\pi(s)-\pi(\beta(s))\otimes+1\otimes\pi(\beta(s))=
    0.
 \]
\end{proof}

\begin{lemma}
\label{lemmaphi0_producto}
 Given $a,c\in Q_{\geq0}$ and $p=\sum_{i=1}^n\lambda_ip_i\in kQ$, with 
 $p_i\in Q_{\geq0}$ for all $i$, we obtain the equality
 \[
  \phi_0(apc) = \phi_0(a)\pi(pc) + \pi(a)\phi_0(p)\pi(c) + \pi(ap)\phi_0(c).
 \]
\end{lemma}

 The proof is immediate using the definition of $\phi_0$ and $k$-linearity of
 $\phi_0$ and $\pi$.
 
\medskip

 Next we prove the last of the preparatory lemmas.
 
\begin{lemma}
 \label{lemma:d_1phi_1phi_0}
 Given $p\in Q_{\geq0}$ and a reduction $r=(r_n,\dots,r_1)$, with $r_i$ a 
 basic reduction for all $i$ such that $1\leq i\leq n$, there is an equality
 \[
  d_1(\phi_1(r_1,p))=\phi_0(p)-\phi_0(r(p)).
 \]
\end{lemma}

\begin{proof}
 We will prove the result by induction on $n$. We will denote 
 $r_i=r_{a_i,s_i,c_i}$.
 
 For $n=1$, there are two cases. The first one is when
 $p\neq a_1s_1c_1$. In this situation, $r(p)=r_1(p)=p$, $\phi_1(r_1,p)=0$ and so 
 the equality is trivially true. In the second case, $p=a_1s_1c_1$, 
 $\phi_1(r_1,p)=\pi(a_1)\otimes s_1\otimes\pi(c_1)$ and 
 $r(p)=r_1(p)=a_1\beta(s_1)c_1$. Moreover,
 \[
 \begin{aligned}
  d_1(\phi_1(r_1,p)) + \phi_0(r_1(p)) &= d_1(\pi(a_1)\otimes s_1\otimes\pi(c_1))
    + \phi_0(a_1\beta(s_1)c_1)\\
  &=\pi(a_1)\phi_0(s_1)\pi(c_1) - \pi(a_1)\phi_0(\beta(s_1))\pi(c_1) 
    + \phi_0(a_1\beta(s_1)c_1).
 \end{aligned}
 \]
 Using Lemma \ref{lemmaphi0_producto}, the last term equals
 \[
  \phi_0(a_1)\pi(\beta(s_1)c_1) + \pi(a_1)\phi_0(\beta(s_1))\pi(c_1) 
    +\pi(a_1\beta(s_1))\phi_0(c_1),
 \]
 so the whole expression is
 \[
 \begin{aligned}
  &\pi(a_1)\phi_0(s_1)\pi(c_1) + \phi_0(a_1)\pi(\beta(s_1)c_1) 
  + \pi(a_1\beta(s_1))\phi_0(c_1)\\
  &=\pi(a_1)\phi_0(s_1)\pi(c_1)+\phi_0(a_1)\pi(s_1c_1)+\pi(a_1s_1)\phi_0(c_1),
 \end{aligned}
 \]
 and using again Lemma \ref{lemmaphi0_producto}, this equals 
$\phi_0(p)$.
 
 Suppose the result holds for $n-1$. As usual, we denote 
 $r'=(r_n,\dots,r_2)$. 
 
 Since $r(p) = r'(r_1(p))$,
 \[
  \begin{aligned}
   d_1(\phi_1(r,p))+\phi_0(r(p)) &= d_1(\phi_1(r_1,p)) + d_1(\phi_1(r',r_1(p)))
     +\phi_0(r'(r_1(p)))\\
   &=d_1(\phi_1(r_1,p)) + \phi_0(r_1(p))\\
   &=\phi_0(p).
  \end{aligned}
 \]
\end{proof}

 
 Consider now an element $p\in\A_2$. By definition we write $p=u_0u_1u_2=
 v_2v_1v_0$ where $u_0u_1$ and $v_1v_0$ are 
 paths in $\mathcal{A}_1$ dividing
 $p$. Suppose $r=r_{a,s,c}$ is a basic reduction such that $r(p)\neq p$. We
 deduce that either $s=u_0u_1$ or $s=v_1v_0$. For an arbitrary reduction 
 $r=(r_n,\dots,r_1)$, we will say that $r$ \textit{starts on the left of $p$}
 if $r_1=r_{a,s,c}$, $s=u_0u_1$ and $asc=p$, and we will say that $r$ 
 \textit{starts on the right of $p$} if $r_1=r_{a,s,c}$, $s=v_1v_0$ and 
 $asc=p$.
 
\begin{proposition}
\label{prop:principi_resolu}
 Let $\{r^p\}_{p\in\A_2}$ and 
 $\{t^p\}_{p\in\A_2}$ be two sets of reductions such that $r^p(p)$ and 
 $t^p(p)$ belong to $k\B$, $r^p$ starts on the left of $p$ and $t^p$ starts
 on the right of $p$.
 Consider 
 $d_2:A\otimes_Ek\A_2\otimes_EA\to A\otimes_Ek\A_1\otimes_EA$ the map of 
 $A$-bimodules defined by $d_2(1\otimes p\otimes1)=\phi_1(t^p,p)
 -\phi_1(r^p,p)$.
 
 The sequence
 {\small
 \[
  \xymatrix{ A\otimes_Ek\A_2\otimes_EA\ar[r]^-{d_2}&
             A\otimes_Ek\A_1\otimes_EA\ar[r]^-{d_1}&
             A\otimes_Ek\A_0\otimes_EA\ar[r]^-{\delta_0}&
             A\otimes_EA\ar[r]^-{\delta_{-1}}&
             A\ar[r]&
             0}
 \]}
 is exact.
\end{proposition}

\begin{proof}
 To check that $d_2$ is well defined, consider the map $\tilde d_2:A\times k\A_2\times A \to A\otimes_Ek\A_1\otimes_EA$ defined by
 $\tilde d_2(x,p,y)=x\phi_1(t^p,p)y-x\phi_1(r^p,p)y$, for all $x,y\in A$, which is clearly multilinear; taking
 into account the definition of $\phi_1$, it is such that 
 $\tilde d_2(xe,p,y)= \tilde d_2(x,ep,y)$ and $\tilde d_2(x,pe,y)=\tilde d_2(x,p,ey)$ for all $e\in E$, so it induces $d_2$ on $A\otimes_Ek\A_2\otimes_EA$.
 
 The sequence is a complex:
 \begin{itemize}
  \item $\delta_{-1}\circ\delta_0=0$ and $\delta_0\circ d_1=0$ follow from
   Lemma \ref{lemma:delta0d1igual0}.
  \item Given $p\in\A_2$, $d_1(d_2(1\otimes p\otimes 1)) = 
    d_1(\phi_1(t^p,p)-\phi_1(r^p,p))$. Using Lemma \ref{lemma:d_1phi_1phi_0},
    this last expression equals $\phi_0(p)-\phi_0(t^p(p)) 
    - \phi_0(p) +\phi_0(r^p(p))$, which is, by Remark 
    \ref{remark:beta}, equal to $-\phi_0(\beta(p))+\phi_0(\beta(p))$,
    so $d_1\circ d_2=0$.
 \end{itemize}
 
 It is exact:
 \begin{itemize}
  \item We already know that this is true at $A$ and at $A\otimes_EA$.
  \item Given $s\in\A_1$, $d_1(1\otimes s\otimes 1)
  -\delta_1(1\otimes s\otimes 1)$ belongs to $\llrr{\overline\LL_0^\prec(s)}_k$:
  indeed, notice  that $\delta_1(1\otimes s\otimes1)=\phi_0(s)$, and 
  $\phi_0(\beta(s))$ belongs to $\llrr{\overline\LL_0^\prec(s)}_k$ since
  $\beta(s)\prec s$. It follows that 
  \[
   d_1(1\otimes s\otimes 1)-\delta_1(1\otimes s\otimes 1)=-\phi_0(\beta(s))
   \in\llrr{\overline\LL_0^\prec(s)}_k.
  \]
  \item Given $p\in\A_2$, we will now prove that 
  $(d_2-\delta_2)(1\otimes p\otimes 1)$ belongs to 
  $\llrr{\overline\LL_1^\prec(p)}_k$. We may write $p=u_0u_1u_2=v_2v_1v_0$, as
  we did just before this proposition and thus  
  $\delta_2(1\otimes p\otimes 1)=\pi(v_2)\otimes v_1v_0\otimes 1 
  - 1\otimes u_0u_1\otimes\pi(u_2)$. Besides, if $r^p=(r_n,\dots,r_1)$ and
  $t^p=(t_m,\dots,t_1)$ with $t_i$ and $r_j$ basic reductions, the fact that
  $r^p$ starts on the left and $t^p$ starts on the right of $p$ gives
  \[
   (d_2-\delta_2)(1\otimes p\otimes 1)=\phi_1(t'^p,t_1(p))-\phi_1(r'^p,r_1(p)),
  \]
  where $t'^p=(t_m,\dots,t_2)$ and $r'^p=(r_n,\dots,r_2)$.
  Since $t_1(p)\prec p$ and $r_1(p)\prec p$, Lemma \ref{lemma:phi_1menores}
  allows us to deduce the result.
 \end{itemize}
Finally, Theorem \ref{teo1} implies that the sequence considered is exact.
\end{proof}

\begin{remark}
 Given $a\in\A_0=Q_1$, we have that 
$\overline\LL_{-1}^\prec(a)=\emptyset$, so for any 
 morphism of $A$-bimodules ${d:A\otimes_Ek\A_0\otimes_EA\to 
 A\otimes_Ek\A_{-1}\otimes_EA}$ such that $(d-\delta_0)(1\otimes a\otimes 
1)$
 belongs to $\llrr{\overline\LL_{-1}^\prec(a)}_k$, it must be $d=\delta_0$.
 
 On the other hand, given $s\in\A_1$, write 
$\beta(s)=\sum_{i=1}^m\lambda_ib_i$. Let
 $r=r_{a,s',c}$ be a basic reduction such that $r(s)\neq s$. We must have 
 $s'=s$ and $a,c\in Q_0$ must coincide with the source and target of $s$, 
 respectively. In other words, the only basic reduction such that $r(s)\neq s$
 is $r_{a,s,c}$ with $a$ and $c$ as we just said, and in this case 
 $r(s)=\beta(s)\in k\B$.
 
 In this situation
 \[
  \{\lambda q\in k^\times Q_{\geq0}: \lambda q\prec s\} 
  = \{\lambda_1b_1,\dots,\lambda_mb_m\},
 \]
 and writing $b_i=b_i^{n_i}\cdots b_i^1$ with $b_i^j\in Q_1$,
 \[
  \overline\LL_0^\prec(s)
  =\bigcup_{i=1}^N\{\lambda_i\pi(b_i^{n_i}\cdots b_i^2)\otimes b_i^1\otimes 
1,
                    \dots,
              \lambda_i\otimes b_i^{n_i}\otimes\pi(b_i^{n_i-1}\cdots b_i^1)\}.
 \]
 If $d:A\otimes_Ek\A_1\otimes_EA\to A\otimes_Ek\A_0\otimes_EA$ verifies
 $(d-\delta_1)(1\otimes s\otimes 1)\in\overline\LL_0^\prec(s)$ and 
 $\delta_0\circ d(s)=0$ for all $s\in \A_1$, then there exists
 $\gamma_i^j\in k$ such that
 \[
  d(1\otimes s\otimes 1)
  =\phi_0(s)-\sum_{i=1}^m\sum_{j=1}^{n_i}
  \gamma_i^j\lambda_i\pi(b_i^{n_i}\cdots b_i^{j+1})\otimes b_i^j\otimes\pi(b_i^{j-1}\cdots b_i^1).
 \]
 From this, applying $\delta_0$ and reordering terms we can deduce that 
 $\gamma_i^j=1$ for all $i,j$. We conclude that the unique morphism with 
 the desired properties is $d_1$.
\end{remark}


\section{Examples}
\label{s:Examples}

In this section we construct explicitly projective 
bimodule resolutions of some algebras using the methods we developed in previous sections.

Given an algebra $A=kQ/I$, we proved in Lemmas \ref{lemma:existencia-sr1} and \ref{lemma:existencia-sr2} that it is
always possible to construct a reduction system $\mathcal R$ such that every
path is reduction-unique. However, it is  not always easy to follow the 
prescriptions given by these lemmas for a concrete algebra. Moreover, the reduction system obtained from a $deglex$ order $\le_{\omega}$ may be sometimes less convenient than other ones.
In fact, describing the set $\mathsf{tip}(I)$ is not in general an easy task.

Bergman's Diamond Lemma is the tool we use to effectively compute a reduction
system in most cases. Next we sketch this procedure, which is also described in 
\cite{B}, Section $5$.

The two sided ideal $I$ is usually presented giving a set $\{x_i\}_{i\in\Gamma}\subseteq kQ$ of generating relations. 
If we fix a well-order on 
$Q_0\cup Q_1$, a function $\omega:Q_1\to\NN$ and consider the total order $\leq_\omega$
on $Q_{\geq0}$, we can easily write $x_i = s_i-f_i$, and we can eventually rescale $x_i$ so that $s_i$ is monic, with 
$s_i>_{\omega} f_i$ 
for all $i$ and define the reduction system 
$\mathcal R=\{(s_i,f_i)\}_{i\in\Gamma}$. Every path $p$ will be reduction-finite
with respect to $\mathcal R$. Bergman's Diamond Lemma says that every path is
reduction-unique if and only if 
for every path $p\in\A_2$ there are reductions $r,t$ with $r$ starting on the left and $t$ starting on the right of $p$ such that $r(p)=t(p)$. 
This last situation is described by saying that $p$ is \textit{resolvable}.
The set $\A_2$ is usually finite and so there is a finite number of conditions
to check. 

In case there exists a non resolvable ambiguity $p\in\A_2$, choose any two reductions $r,t$ starting on the left and on the right respectively 
with $r(p)$ and $t(p)$ both
irreducible. The element $r(p)-t(p)$ belongs to $I\setminus\{0\}$. We can write $r(p)-t(p) = s-f$ 
with 
$f<_\omega s$ 
and add the element $(s,f)$ to our reduction system,
and so $p$ is now resolvable. New ambiguities may now appear, so it is necessary to iterate this process, which may have infinitely many steps, but we will arrive to a 
reduction system $\mathcal R$ satisfying condition ($\lozenge$).

Next we give an example to illustrate this procedure, which will be also useful to exhibit a case where another 
reduction system found in an alternative way is better that the prescribed one.

\begin{example}
\label{ej:cubica}
Consider 
the algebra of Example \ref{example:otroorden}. Let $x<y<z$ and 
$\omega(x)=\omega(y)=\omega(z)=1$.
The ideal $I$ is presented as the two sided ideal generated by the element $x^3 + y^3 + z^3 - xyz$. We see that
$z^3 = \mathsf{tip}(z^3 - (xyz -x^3 - y^3))$, so we start considering
the reduction system $\mathcal R = \{(z^3,xyz - x^3 - y^3)\}$. Notice 
that $\A_2=\{z^4\}$. If we apply the reduction $r_{z,z^3,1}$ to $z^4$ we obtain
$zxyz - zx^3 - zy^3$ which is irreducible. On the other hand, if we apply
the reduction $r_{1,z^3,z}$ to $z^4$ we obtain $xyz^2 - x^3z - y^3z$ which is also irreducible and different from the first one. The difference between 
them is $xyz^2-x^3z-y^3z-zxyz+zx^3+zy^3$, so we add $(xyz^2, x^3z+y^3z +zxyz-zx^3-zy^3)$ to the reduction system $\mathcal R$. Notice that now the set $\A_2$ is 
$\{z^4, xyz^3\}$. Applying reductions on the left and on the right to the element
$xyz^3$ we obtain again two different irreducible elements and, proceeding as
before, we see that we have to add the element $(y^3z^2, -x^3z^2-z^2xyz+z^2x^3+z^2y^3+xyxyz-xyx^3-xy^4)$ to our reduction system $\mathcal R$. We obtain the new ambiguity $y^3z^3$ which is not difficult to see that it is
resolvable. Thus, the reduction system
\[
\begin{aligned}
\mathcal R_1=\{&(z^3,xyz - x^3 - y^3),(xyz^2, x^3z+y^3z +zxyz-zx^3-zy^3),\\
&(y^3z^2, -x^3z^2-z^2xyz+z^2x^3+z^2y^3+xyxyz-xyx^3-xy^4)\},
\end{aligned}
\]
satisfies condition $(\lozenge)$.

There is another
reduction system for this algebra, namely  $\mathcal R_2 = \{(xyz,x^3+y^3+z^3)\}$. 
Let us denote $\A_n^1$ and $\A_n^2$ the respective set of $n$-ambiguities. 
Notice that $z^{\frac{3}{2}(n+1)}\in \A_n^1$ for $n$ odd and $z^{\frac{3}{2}n+1}\in\A_n^1$ for $n$ even, 
so $\A_n^1$ is not empty for all $n\in\NN$. On the other hand, $\A_n^2$ is empty for all $n\geq2$. 
We conclude that using $\mathcal R_2$ we will obtain a resolution
of length $2$, with differentials given explicitely by Proposition 
\ref{prop:principi_resolu}, and using $\mathcal R_1$ the resolution 
obtained will have infinite length. 
This shows how different can the resolutions from different reduction 
systems be.

Notice that $\mathcal R_2$ cannot be obtained
by the procedure described above by any choice of order on $Q_0\cup Q_1$ 
and weight $\omega$.
The algebra $A=k<x,y,z>/(xyz-x^3-y^3-z^3)$ is in fact a $3$-Koszul 
algebra. Indeed, denoting by $V$ the $k$-vector space spanned by $x,y,z$ 
and by $R$ the one dimensional $k$-vector space spanned by the relation 
$xyz-x^3-y^3-z^3$, it is straightforward that 
\[
 R\otimes V\otimes V\cap V\otimes V\otimes R = \{0\},
\]
and so the intersection is a subset of $V\otimes R\otimes V$. Theorem 2.5 
of \cite{Be1} guarantees that $A$ is $3$-Koszul.

The resolution we obtain from the reduction system $\mathcal R_2$ is the 
Koszul resolution, since it is minimal, see Theorem \ref{teo3}. As we shall see, this is a 
particular case of a general situation.
\end{example}

\subsection{The algebra counterexample to Happel's question}

 Let $\xi$ be an element of the field $k$ and let $A$ be the $k$-algebra with generators
 $x$ and $y$, subject to the relations $x^2=0=y^2$, $yx=\xi xy$.
 Choose the order $x<y$ with weights $\omega(x)=\omega(y)=1$ and fix
 the reduction system $\mathcal R=\{(x^2,0),(y^2,0),(yx,\xi xy)\}$.
 The set $\B$ of irreducible paths is thus $\{1,x,y,xy\}$. It is easy
 to verify that $\A_2=\{x^3,yx^2,y^2x,y^3\}$ and that all paths in $\A_2$ are 
 reduction-unique. Bergman's Diamond Lemma guarantees that $\mathcal R$ 
 satisfies $(\lozenge)$.
 
 The only path of length $2$ not in $S$ is $xy$; Proposition
\ref{cuadratico}
 implies that 
 for each $n$, $\A_n$ is the set of paths of lenght $n+1$ not divisible by 
 $xy$,
 \[
  \A_n=\{y^sx^t:s+t=n+1\}.
 \]
 
\begin{lemma}
 The following complex provides the beginning of an $A$-bimodule projective resolution
 of the algebra $A$
 {\small
 \[
  \xymatrix{A\otimes_Ek\A_2\otimes_EA\ar[r]^-{d_2}&
            A\otimes_Ek\A_1\otimes_EA\ar[r]^-{d_1}&
            A\otimes_Ek\A_0\otimes_EA\ar[r]^-{\delta_0}&
            A\otimes_EA\ar[r]^-{\delta_{-1}}&
            A\ar[r]&
            0}
 \]}
 where $d_1$ is the $A$-bimodule map such that
 {\small
 \[
  \begin{aligned}
   &d_1(1\otimes x^2\otimes 1) = x\otimes x\otimes 1+1\otimes x\otimes x,\\
   &d_1(1\otimes y^2\otimes 1) = y\otimes y\otimes 1+1\otimes y\otimes y,\\
   &d_1(1\otimes yx\otimes 1) = y\otimes x\otimes 1 + 1\otimes y\otimes x
      -\xi x\otimes y\otimes 1 - \xi\otimes x\otimes y
  \end{aligned}
 \]}
 and $d_2$ is the $A$-bimodule morphism such that
 {\small
 \[
 \begin{aligned}
  &d_2(1\otimes y^3\otimes 1) = y\otimes y^2\otimes 1 - 1\otimes y^2\otimes y,\\
  &d_2(1\otimes y^2x\otimes 1) = y\otimes yx\otimes 1 +\xi\otimes yx\otimes y
      +\xi^2x\otimes y^2\otimes 1 - 1\otimes y^2\otimes x,\\
  &d_2(1\otimes yx^2\otimes 1) = y\otimes x^2\otimes 1 - 1\otimes yx\otimes x
      -\xi x\otimes yx\otimes 1 - \xi^2\otimes x^2\otimes y\\
  &d_2(1\otimes x^3\otimes 1) = x\otimes x^2\otimes 1-1\otimes x^2\otimes x.
 \end{aligned}
 \]}
\end{lemma}

\begin{proof}
 We apply Proposition \ref{prop:principi_resolu} to the following sets
 $\{r^p\}_{p\in\A_2}$ of left reductions and $\{t^p\}_{p\in\A_2}$ of right
 reductions, where
 \[
  \begin{aligned}
   &r^{y^3} = r_{1,y^2,y}, 
     & &r^{y^2x}=r_{1,y^2,x},
     & & &t^{y^3} = r_{y,y^2,1},
     & & &t^{y^2x} = (r_{x,y^2,1},r_{1,yx,y},r_{y,yx,1}),\\
   &r^{yx^2}=(r_{1,x^2,y},r_{x,yx,1},r_{1,yx,x}),
     & &r^{x^3}=r_{1,x^2,x},
     & & &t^{yx^2} = r_{y,x^2,1},
     & & &t^{x^3} = r_{x,x^2,1}.
  \end{aligned}
 \]
\end{proof}
 One can find an $A$-bimodule resolution of $A$  in \cite{BGMS} 
 and in \cite{BE}; the authors also compute the Hochschild cohomology of $A$ therein. We recover this resolution
 with our method.
 
 Given $q\in\A_n$, there are $s,t\in\NN$ such that $s+t=n+1$ and $q=y^sx^t$.
 Suppose $q=apc$ with $p=y^{s'}x^{t'}\in\A_{n-1}$ and $a,c\in Q_{\geq0}$. 
 Since $s+t=n+1$ and $s'+t'=n$, either $a$
 belongs to $Q_0$ and $c=x$ or $a=y$ and $c\in Q_0$. As a consequence of this
 fact, the maps 
 $\delta_n:kQ\otimes_Ek\A_n\otimes_EkQ\to kQ\otimes_Ek\A_{n-1}\otimes_EA$ are
 \[
  \delta_n(1\otimes y^sx^t\otimes 1) = 
     \begin{dcases*}
        y\otimes y^{s-1}x^t\otimes 1+(-1)^{n+1}\otimes y^sx^{t-1}\otimes x, & 
           if $s\neq0$ and $t\neq0$,\\
        y\otimes y^n\otimes 1 + (-1)^{n+1}\otimes y^n\otimes y, &
           if $t=0$,\\
        x\otimes x^n\otimes 1 + (-1)^{n+1}\otimes x^n\otimes x, &
           if $s=0$,
     \end{dcases*}
 \]
 Moreover, given a basic reduction $r=r_{a,s,c}$, the fact that $s$ belongs to
 $S=\{x^2,y^2,yx\}$ implies that $r(y^sx^t)$ is either $0$ or 
 $\xi y^{s-1}xyx^{t-1}$. Considering 
 the reduction system 
 $\mathcal R$,
 if $s\neq0$ and $t\neq0$, then
 \[
  \overline\LL_{n-1}^\prec(y^sx^t)=
  \{\xi^sx\otimes y^sx^{t-1}\otimes 1, \xi^t\otimes y^{s-1}x^t\otimes y\}.
 \]
 In case $s=0$ or $t=0$, the set $\overline\LL_{n-1}^\prec(y^sx^t)$ is empty.
 
 \medskip
 
 The computation of $d_2-\delta_2$ suggests the definition of the maps
 \[
 d_n:A\otimes_Ek\A_n\otimes_EA\to A\otimes_Ek\A_{n-1}\otimes_EA
 \]
  as follows
 \[
  d_n(1\otimes y^sx^t\otimes 1) = \delta_n(1\otimes y^sx^t\otimes 1) + 
                                  \epsilon(\xi^sx\otimes y^sx^{t-1}\otimes 1
                                  +\xi^t\otimes y^{s-1}x^t\otimes y)
 \]
 where $\epsilon$ denotes a sign depending on $s,t,n$. The equality 
 $d_{n-1}\circ d_n=0$ shows that making the choice
 $\epsilon = (-1)^s$ does the job.
 
 Finally, Theorem \ref{teo1} shows that the complex
{\small
 \[
  \xymatrix{\cdots\ar[r] &
            A\otimes_Ek\A_n\otimes_EA\ar[r]^-{d_n}&
            \cdots\ar[r]^-{d_1}&
            A\otimes_Ek\A_0\otimes_EA\ar[r]^-{d_0}&
            A\otimes_EA\ar[r]^-{d_{-1}}&
            A\ar[r]&
            0}
 \]}
 with
 {\small
 \[
  d_n(1\otimes y^sx^t\otimes 1) = y\otimes y^{s-1}x^t\otimes 1 + 
  (-1)^{n+1}1\otimes y^sx^{t-1}\otimes x 
  + (-1)^s\xi^sx\otimes y^sx^{t-1}\otimes 1 + 
  (-1)^s\xi^t\otimes y^{s-1}x^t\otimes y,
 \]}
 for $s>0$ and $t>0$,
 and
 {\small\[
  \begin{aligned}
   &d_n(1\otimes y^{n+1}\otimes1) = y\otimes y^n\otimes 1 + (-1)^{n+1}1\otimes y^n\otimes y,\\
   &d_n(1\otimes x^{n+1}\otimes1) = x\otimes x^n\otimes 1 + (-1)^{n+1}1\otimes x^n\otimes x,
  \end{aligned}
 \]}
 is a projective bimodule resolution of $A$.
 
 Again, the algebra $A$ is Koszul, see for example \cite{Be2} and the 
resolution obtained using our procedure is the Koszul resolution, which is 
the minimal one, see Theorem \ref{teo3}.
 
\subsection{Quantum complete intersections}
\label{ej:qci}

 These algebras generalize the previous case. Instead of the relations 
 $x^2=0=y^2$, $yx=\xi xy$, we have $x^n=0=y^m$, $yx = \xi xy$, where $n$ and
 $m$ are fixed positive integers, $n,m > 1$.

 We still denote the algebra by $A$.
 Consider the order $x<y$ with weights $\omega(x)=\omega(y)=1$. The 
 set of $2$-ambiguities associated to the reduction system $\mathcal R=\{(x^n,0),(y^m,0),(yx,\xi xy)\}$ is 
 $\A_2=\{y^{m+1},y^mx,yx^n,x^{n+1}\}$, and the set of irreducible paths is 
 $\B=\{x^iy^j\in k\llrr{x,y} : 0\leq i\leq n-1, 0\leq j\leq m-1\}$. We easily check that every path in
 $\A_2$ is reduction-unique and using Bergman's Diamond Lemma, we conclude that $\mathcal R$ satisfies $(\lozenge)$, 
 Also,
 $\A_1=S=\{y^m,yx,x^n\}$
 and
 $\A_3=\{y^{2m},y^{m+1}x,y^mx^n,yx^{n+1},x^{2n}\}$.

 Denote by $\varphi:\NN_0^2\to\NN_0$ the map
 \[
  \varphi(s,n) =
  \begin{dcases*}
   \frac{s}{2}n & if $s$ is even,\\
   \frac{s-1}{2}n + 1 & if $s$ is odd.
  \end{dcases*}
 \]
 Given $N\in\NN$, the set of $N$-ambiguities is
 $\A_N=\{y^{\varphi(s,m)}x^{\varphi(t,n)}:s+t=N+1\}$. We will sometimes write
$(s,t)$ instead of $y^{\varphi(s,m)}x^{\varphi(t,n)} \in \A_N$.
 
 We first compute the beginning of the resolution.
 
\begin{lemma}
 The following complex provides the beginning of a projective resolution of
 $A$ as $A$-bimodule:
 {\small
 \[
  \xymatrix{A\otimes_Ek\A_2\otimes_EA\ar[r]^-{d_2}&
            A\otimes_Ek\A_1\otimes_EA\ar[r]^-{d_1}&
            A\otimes_Ek\A_0\otimes_EA\ar[r]^-{\delta_0}&
            A\otimes_EA\ar[r]^-{\delta_{-1}}&
            A\ar[r]&
            0}
 \]}
 where $d_1$ and $d_2$ are morphisms of $A$-bimodules given by the formulas
 {\small
 \[
  \begin{aligned}
    &d_1(1\otimes x^n\otimes 1) = \sum_{i=0}^{n-1}x^i\otimes x\otimes x^{n-1-i},\\
    &d_1(1\otimes y^m\otimes 1) = \sum_{i=0}^{m-1}y^i\otimes y\otimes y^{m-1-i},\\
    &d_1(1\otimes yx\otimes 1) = 1\otimes y\otimes x + y\otimes x\otimes 1 
        -\xi\otimes x\otimes y - \xi x\otimes y\otimes 1\\
    &d_2(1\otimes y^{m+1}\otimes 1)= y\otimes y^m\otimes 1 - 1\otimes y^m\otimes y, \\
    &d_2(1\otimes y^mx\otimes 1)=\sum_{i=0}^{m-1}\xi^iy^{m-1-i}\otimes yx\otimes y^i 
        + \xi^mx\otimes y^m\otimes 1 
        - 1\otimes y^m\otimes x\\
    &d_2(1\otimes yx^n\otimes 1) = y\otimes x^n\otimes 1 
         -\sum_{i=0}^{n-1}\xi^i x^i\otimes yx\otimes x^{n-1-i}
         -\xi^n\otimes x^n\otimes y,\\
    &d_2(1\otimes x^{n+1}\otimes 1)=x\otimes x^n\otimes 1 - 1\otimes x^n\otimes x.
  \end{aligned}
 \]}
\end{lemma}
 
\begin{proof}
 It is straightforward, using Proposition \ref{prop:principi_resolu} applied to
 the set $\{r^p\}_{p\in\A_2}$ of left reductions, where
 \[
  \begin{aligned}
   &r^{y^{m+1}}=r_{1,y^m,y}, & &r^{y^mx}=r_{1,y^m,x},\\
   &r^{yx^n}=(r_{1,x^n,y},\dots,r_{x,yx,x^{n-2}},r_{1,yx,x^{n-1}}) &
      &r^{x^{n+1}}=r_{1,x^n,x},
  \end{aligned}
 \]
 and the set $\{t^p\}_{p\in\A_2}$ of right reductions, where
 \[
  \begin{aligned}
   &t^{y^{m+1}}=r_{y,y^m,1}, & 
        &t^{y^mx}=(r_{x,y^m,1},\dots,r_{y^{m-2},yx,y},r_{y^{m-1},yx,1}),\\
   &y^{yx^n}=r_{y,x^n,1}, & &t^{x^{n+1}}=r_{x,x^n,1}.
  \end{aligned}
 \]
\end{proof}
 Of course we want to construct the rest of the resolution. Denote 
$(s,t)=y^{\varphi(s,m)}x^{\varphi(t,n)}\in\A_N$. We will 
first describe the set $\overline\LL_{N-1}^\prec(s,t)$. There are four 
cases, depending on the parity of $s,t$ and $N$. With this in view, it 
is useful to make some previous computations that we list below.
 \begin{enumerate}
  \item For $s$ even, for all $j$, $0\leq j\leq m-1$, 
    $y^{\varphi(s,m)}=y^{m-1-j}y^{\varphi(s-1,m)}y^j$.
  \item For $s$ odd, $y^{\varphi(s,m)}=yy^{\varphi(s-1,m)}=
        y^{\varphi(s-1,m)}y.$
  \item For $t$ even, for all $i$, $0\leq i\leq n-1$, $x^{\varphi(t,n)} =
        x^ix^{\varphi(t-1,n)}x^{n-i-1}$,
  \item For $t$ odd, $x^{\varphi(t,n)}=xx^{\varphi(t-1,n)}=x^{\varphi(t-1,n)}x$.
 \end{enumerate}

\smallskip

 \begin{description}
  \item[First case] $N$ even, $s$ even, $t$ odd,
  {\small
   \[
    \overline\LL_{N-1}^\prec(s,t) = \{\xi^{\varphi(t,n)j}y^{m-1-j}\otimes (s-1,t)\otimes y^j\}_{j=1}^{m-1}\cup\{\xi^{\varphi(s,m)}x\otimes(s,t-1)\otimes1\}.
   \]}
  \item[Second case] $N$ even, $s$ odd, $t$ even,
  {\small
  \[
   \overline\LL_{N-1}^\prec(s,t) = \{\xi^{\varphi(t,n)}\otimes(s-1,t)\otimes y\}\cup\{\xi^{\varphi(s,m)i}x^i\otimes(x,t-1)\otimes x^{n-1-i}\}_{i=1}^{n-1}.
  \]}
  \item[Third case] $N$ odd, $s$ even, $t$ even,
  {\small
  \[
   \overline\LL_{N-1}^\prec(s,t) = \{\xi^{\varphi(t,n)j}y^{m-1-j}\otimes(s-1,t)\otimes y^j\}_{j=1}^{m-1}\cup\{\xi^{\varphi(s,m)i}x^i\otimes(s,t-1)\otimes x^{n-1-i}\}_{i=1}^{n-1}.
  \]}
  \item[Fourth case] $N$, $s$ and $t$ odd,
  {\small
  \[
   \overline\LL_{N-1}^\prec(s,t)=\{\xi^{\varphi(t,n)}1\otimes(s-1,t)\otimes y,\xi^{\varphi(s,m)}x\otimes(s,t-1)\otimes 1\}.
  \]}
 \end{description} 
\begin{remark}
 We observe that, analogously to the case $n=m=2$,
 \[
  \begin{aligned}
   (d_1-\delta_1)(1\otimes(s,t)\otimes1) 
      = (-1)^s\sum_{u\in\overline\LL_0^\prec(s,t)}u,\\
   (d_2-\delta_2)(1\otimes(s,t)\otimes1)
      = (-1)^s\sum_{u\in\overline\LL_1^\prec(s,t)}u.
  \end{aligned}
 \]
 Proposition \ref{propfinal} for $R=\ZZ$ 
 guarantees that there exist $A$-bimodule maps 
 $d_N:A\otimes_Ek\A_N\otimes_EA\to A\otimes_Ek\A_{N-1}\otimes_EA$ such that
 $(d_N-\delta_N)(1\otimes(s,t)\otimes1)\in
   \llrr{\overline\LL_{N-1}^\prec(s,t)}_\ZZ$ and, most important, the complex
 $(A\otimes_Ek\A_\bullet\otimes_EA,d_\bullet)$ is a projective resolution of
 $A$ as $A$-bimodule.
 
 We are not yet able at this point to give the explicit formulas of the 
 differentials.
 
 In order to illustrate the situation, let us describe what happens for $N=3$.
 We know after the mentioned proposition that there exist $t_1,t_2\in\ZZ$ 
 such that
 \[
 \begin{aligned}
  d_3(1\otimes y^{m+1}x\otimes 1)&=d_3(1\otimes(3,1)\otimes 1)\\
  &=\delta_3(1\otimes(3,1)\otimes1)+t_1\xi\otimes(2,1)\otimes y + 
    t_2\xi^3x\otimes(3,0)\otimes1\\
  &=y\otimes y^mx\otimes 1 + 1\otimes y^{m+1}\otimes x + 
    t_1\xi\otimes y^mx\otimes y + t_2\xi^3 x\otimes y^{m+1}\otimes 1.
 \end{aligned}
 \]
 Of course, $d_2\circ d_3=0$. It follows from this equality 
 that $t_1=t_2=-1$. This example motivates the following lemma, stated in terms of the preceding notations.
\end{remark}

\begin{lemma}
 The $A$-bimodule morphisms $d_N:A\otimes_Ek\A_N\otimes_EA\to
 A\otimes_Ek\A_{N-1}\otimes_EA$ defined by the formula 
 \[
  d_N(1\otimes(s,t)\otimes1)=\delta_N(1\otimes(s,t)\otimes1) +
    (-1)^s\sum_{u\in\overline\LL_{N-1}^\prec(s,t)}u
 \]
 satisfy the hypotheses of Thm. \ref{teo1}.
\end{lemma}
\begin{proof}
 It is straightforward.
\end{proof}

We gather all the information we have obtained about the projective bimodule
resolution of $A$ in the following proposition.

\begin{proposition}
 The complex of $A$-bimodules $(A\otimes_Ek\A_\bullet\otimes_EA,d_\bullet)$, 
 with 
 \[
 \A_N=\{y^{\varphi(s,m)}x^{\varphi(t,n)}:s+t=N+1\}
 \]
  and differentials
 defined as follows is exact.
 \begin{enumerate}
  \item For $N$ even, $s$ even and $t$ odd,
  {\small\[
  \begin{aligned}
   d_N(1\otimes(s,t)\otimes1)&=y^{m-1}\otimes(s-1,t)\otimes1 
   +\sum_{j=1}^{m-1}(-1)^s\xi^{\varphi(t,n)j}y^{m-1-j}\otimes(s-1,t)\otimes y^j\\
   &+(-1)^{N+1}1\otimes(s,t-1)\otimes x
   +(-1)^s\xi^{\varphi(s,m)}x\otimes(s,t-1)\otimes 1.
  \end{aligned}
  \]}
  \item For $N$ even, $s$ odd and $t$ even,
  {\small\[
  \begin{aligned}
   d_N(1\otimes(s,t)\otimes1)&=y\otimes(s-1,t)\otimes1 
   + (-1)^s\xi^{\varphi(t,n)}\otimes(s-1,t)\otimes y\\
   &+(-1)^{N+1}1\otimes(s,t-1)\otimes x^{n-1} + \sum_{i=1}^{n-1}(-1)^s\xi^{\varphi(s,m)i}x^i\otimes(s,t-1)\otimes x^{n-1-i}
  \end{aligned}
  \]}
  \item For $N$ odd, $s$ and $t$ even,
  {\small\[
  \begin{aligned}
   d_N(1\otimes(s,t)\otimes 1)&=y^{m-1}\otimes(s-1,t)\otimes1 +
   \sum_{j=1}^{m-1}(-1)^s\xi^{\varphi(t,n)j}y^{m-1-j}\otimes(s-1,t)\otimes y^j\\
   &+(-1)^{N+1}1\otimes(s,t-1)\otimes x^{n-1} + \sum_{i=1}^{n-1}(-1)^s\xi^{\varphi(s,m)i}x^i\otimes(s,t-1)\otimes x^{n-1-i}
  \end{aligned}  
  \]}
  \item For $N$, $s$ and $t$ odd,
  {\small\[
  \begin{aligned}
   d_N(1\otimes(s,t)\otimes1) &=y\otimes(s-1,t)\otimes1 
   + (-1)^s\xi^{\varphi(t,n)}\otimes(s-1,t)\otimes y\\
   &+(-1)^{N+1}1\otimes(s,t-1)\otimes x 
   + (-1)^s\xi^{\varphi(s,m)}x\otimes(s,t-1)\otimes1.
  \end{aligned}     
  \]}
 \end{enumerate}
 \end{proposition}
 Again, we obtain the minimal resolution of $A$, even for $n\neq2$ or $ 
m\neq 2$, when the algebra is not homogeneous.

\subsection{Down-up algebras}
Given $\alpha,\beta,\gamma\in k$, we will denote $A(\alpha,\beta,\gamma)$ the
quotient of $k\llrr{d,u}$ by the two sided ideal $I$ generated by relations
\[
 \begin{aligned}
  &d^2u-\alpha dud - \beta ud^2 - \gamma d = 0,\\
  &du^2-\alpha udu - \beta u^2d - \gamma u=0.
 \end{aligned}
\]
Down-up algebras have been deeply studied  
since they were defined in \cite{BR}. We can mention the articles \cite{CM}, \cite{BW},\cite{BG}, 
 \cite{CS}, \cite{CL}, \cite{KK}, \cite{KMP}, \cite{Ku1}, \cite{Ku2}, \cite{P1},
\cite{P2}, \cite{P3}, in which the authors prove diverse properties of 
down-up algebras.
It is well known that they 
are noetherian if and only if $\beta\neq0$ \cite{KMP}. They are graded with
$\mathsf{dg}(d)=1$, $\mathsf{dg}(u)=-1$, and they are filtered if we consider
$d$ and $u$ of weight $1$. If $\gamma=0$ they are also graded by this weight.

Down-up algebras are $3$-Koszul if $\gamma=0$, and if $\gamma\neq0$, they 
are PBW deformations of $3$-Koszul algebras \cite{BG}.

Little is known about their Hochschild homology and cohomology, except for the center, described in \cite{Z} and \cite{Ku1}.
We apply our methods to construct a projective resolution of $A$ as $A$-bimodule,
and then use this resolution to compute $H^\bullet(A,A^e)$ and prove that 
in the noetherian case, $A(\alpha,\beta,\gamma)$ is $3$-Calabi-Yau if and only 
if $\beta=-1$. Moreover, in this situation we exhibit a potential $\Phi(d,u)$ such that
the relations are in fact the cyclic derivatives $\partial_u\Phi$ and 
$\partial_d\Phi$, respectively.

 We briefly recall that a $d$-Calabi-Yau algebra is an associative algebra
 such that there is an isomorphism $f$ of $A$-bimodules
 \[
  \Ext_{A^e}^i(A,A^e)\cong\begin{dcases*}
                           0 & if $i\neq d$,\\
                           A & if $i=d$.
                          \end{dcases*}
 \]
 where the $A$-bimodule outer structure of $A^e$ is used for the computation
 of $\Ext_{A^e}^i(A,A^e)$, while the isomorphism $f$ takes account of the inner
 bimodule structure of $A^e$. Bocklandt proved in \cite{Bo} that graded Calabi-Yau algebras come from a 
 potential and Van den Bergh \cite{VdB} generalized this result to complete algebras with respect to the $I$-adic
 topology.
 
We fix a lexicographical order such that $d<u$, with weights 
$\omega(d)=1=\omega(u)$. The reduction system  
$\mathcal R =\{(d^2u, \alpha dud + \beta ud^2 + \gamma d),(du^2,\alpha udu + \beta u^2d + \gamma u)\}$
has $\B=\{u^i(du)^kd^j:i,k,j\in\NN_0\}$ as set of irreducible paths and
$\A_2=\{d^2u^2\}$; using Bergman's Diamond Lemma we see that $\mathcal R$ satisfies
condition $(\lozenge)$. Also, $\A_0=\{d,u\}$ and $\A_n=\emptyset$ for all $n\geq3$.
The set $\B$ is the $k$-basis already considered in \cite{BR}. 

The reductions $r^{d^2u^2}=(r_{u,d^2u,1},r_{1,d^2u,u})$ and 
$t^{d^2u^2}=(t_{1,du^2,d},t_{d,du^2,1})$ are respectively left and right 
reductions of $d^2u^2$.

In view of Proposition \ref{prop:principi_resolu} and observing that $\delta_{-1}$ is in fact an epimorphism and that $\A_3=\emptyset$, the following complex 
gives a free resolution of $A$ as $A$-bimodule:
{\small
\[
0 \to A\otimes_Ekd^2u^2\otimes_EA\to ^{d_2} A\otimes_E(kd^2u\oplus kdu^2)\otimes_EA\to ^{d_1} A\otimes_E(kd\oplus ku)\otimes_EA\to ^{\delta_0} A\otimes_EA\to ^{\delta_{-1}} A\to 0
\]
}

where 
 {\small\[
 \begin{aligned}
  d_1(1\otimes d^2u\otimes 1) &= 1\otimes d\otimes du+ d\otimes d\otimes u + d^2\otimes u\otimes 1
  -\alpha(1\otimes d\otimes ud + d\otimes u\otimes d + du\otimes d\otimes 1)\\
  &-\beta(1\otimes u \otimes d^2 + u\otimes d\otimes d + ud\otimes d\otimes 1)
   -\gamma\otimes d\otimes 1,\\
  d_1(1\otimes du^2\otimes 1) &= 1\otimes d\otimes u^2 + d\otimes u\otimes u
  +du\otimes u\otimes 1 - \alpha(1\otimes u\otimes du + u\otimes d\otimes u
  + ud\otimes u\otimes 1)\\
  &-\beta(1\otimes u\otimes ud + u\otimes u\otimes d + u^2\otimes d\otimes 1) - \gamma\otimes u\otimes1,
 \end{aligned}
 \]}
 and
 {\small\[
  \begin{aligned}
    d_2(1\otimes d^2u^2\otimes 1)=d\otimes du^2\otimes 1
    +\beta\otimes du^2\otimes d - 1\otimes d^2u\otimes u 
    - \beta u\otimes d^2u\otimes 1.
  \end{aligned}
 \]}
 As we have proved in general, the map $d_2$ takes into account the reductions
 applied to the ambiguity.

\begin{proposition}
 Suppose that $\beta\neq0$. The algebra $A(\alpha,\beta,\gamma)$ is 
 $3$-Calabi-Yau if and only if $\beta=-1$.
\end{proposition}
\begin{proof}
 We need to compute $\Ext_{A^e}^\bullet(A,A^e)$. We apply the functor 
 $\Hom_{A^e}(-,A^e)$ to the previous resolution,
 and we use
 that for any finite dimensional vector space $V$ which is also an
 $E$-bimodule, the space $\Hom_{A^e}(A\otimes_EV\otimes_EA,A^e)$ is isomorphic to 
 $\Hom_{E^e}(V,A^e)$, and this last one is, in turn, isomorphic to 
 $A\otimes_EV^*\otimes_EA$. All the isomorphisms are natural.
 The explicit expression of the last isomorphism
 is, fixing a $k$-basis $\{v_1,\dots,v_n\}$ of $V$ and its dual 
 basis $\{\varphi_1,\dots,\varphi_n\}$ of $V^*$,
 \[
  \begin{aligned}
   A\otimes_EV^*\otimes_EA&\to\Hom_{E^e}(V,A^e)\\
   a\otimes\varphi\otimes b&\mapsto [v\mapsto \varphi(v)b\otimes a]
  \end{aligned}
 \]
 with inverse $f\mapsto \sum_{i,j}b_j^i\otimes \varphi_i\otimes a_j^i$, where
 $f(v_i)=\sum_ja_j^i\otimes b_j^i$.
 
 After these identifications, we obtain the following complex of $k$-vector
 spaces whose homology is $\Ext_{A^e}^\bullet(A,A^e)$
 
 
 \[
 0\to A\otimes_EA\overset{\delta_0^*}{\to} A\otimes_E(kD\oplus kU)\otimes_EA\overset{d_1^*}{\to}A\otimes_E(kD^2U\oplus kDU^2)\otimes_EA\overset{d_2^*}{\to} A\otimes_EkD^2U^2\otimes_EA\to 0,
 \]
 
 where $\{D,U\}$ denotes the dual basis of $\{d,u\}$ and, accordingly, we 
 denote with capital letters the dual bases of the other spaces .
 
 The maps in the complex are, explicitely:
 {\small\[
  \begin{aligned}
   &\delta_0^*(1\otimes 1) =1\otimes D\otimes d-d\otimes D\otimes1 
       + 1\otimes U\otimes u - u\otimes U\otimes 1\\
   &d_1^*(1\otimes U\otimes 1) = 1\otimes D^2U\otimes d^2 
       -\alpha d\otimes D^2U\otimes d - \beta d^2\otimes D^2U\otimes 1
       +u\otimes DU^2\otimes d \\
       &\hskip2.3cm+ 1\otimes DU^2\otimes du - \alpha du\otimes DU^2\otimes 1
       -\alpha\otimes DU^2\otimes ud - \beta ud\otimes DU^2\otimes 1 \\
       &\hskip2.3cm-\beta d\otimes DU^2\otimes u - \gamma\otimes DU^2\otimes 1.\\
   &d_1^*(1\otimes D\otimes 1) = du\otimes D^2U\otimes 1 
       + u\otimes D^2U\otimes d-\alpha ud\otimes D^2U\otimes 1
       -\alpha\otimes D^2U\otimes du \\
       &\hskip2.3cm- \beta d\otimes D^2U\otimes u-\beta\otimes D^2U\otimes ud
       -\gamma\otimes D^2U\otimes 1 + u^2\otimes DU^2\otimes 1 \\
       &\hskip2.3cm-\alpha u\otimes DU^2\otimes u-\beta\otimes DU^2\otimes u^2.\\
   &d_2^*(1\otimes DU^2\otimes1)=1\otimes D^2U^2\otimes d 
       + \beta d\otimes D^2U^2\otimes 1,\\
   &d_2^*(1\otimes D^2U\otimes 1)=-u\otimes D^2U^2\otimes 1
   -\beta\otimes D^2U^2\otimes u.
  \end{aligned}
 \]}
 Consider the following isomorphisms of $A$-bimodules
 {\small
 \begin{align*}
 &\psi_0:A\otimes_EA\to A\otimes_Ekd^2u^2\otimes_EA,\\ 
    &\hskip1cm\psi_0(1\otimes 1)=1\otimes d^2u^2\otimes 1,\\
 &\psi_1:A\otimes_E(kD\oplus kU)\otimes_EA\to A\otimes_E(kd^2u\oplus kdu^2)\otimes_EA\\
    &\hskip1cm\psi_1(1\otimes D\otimes 1)=1\otimes du^2\otimes 1,\mbox{ and }\psi_1(1\otimes U\otimes 1)=1\otimes d^2u\otimes 1\\
 &\psi_2:A\otimes_E(kD^2U\oplus kDU^2)\otimes_EA\to A\otimes_E(kd\oplus ku)\otimes_EA,\\
    &\hskip1cm\psi_2(1\otimes D^2U\otimes 1) = 1\otimes u\otimes 1,\mbox{ and }\psi_2(1\otimes DU^2\otimes 1)=1\otimes d\otimes1\\
 &\psi_3:A\otimes_E kD^2U^2\otimes_E\to A\otimes_EA\\
    &\hskip1cm\psi_3(1\otimes D^2U^2\otimes 1) = 1\otimes 1.
 \end{align*}
 }
 
 It is straightforward to verify that the following diagram commutes, thus
 inducing isomorphisms between the homology spaces of both horizontal 
 sequences:
 
 {\footnotesize\[
  \xymatrix{
  0\ar[r] & A\otimes_EA\ar[r]^-{\delta_0^*}\ar[d]^-{\psi_0} & A\otimes_E(k\A_0)^*\otimes_EA\ar[r]^-{d_1^*}\ar[d]^-{\psi_1} & A\otimes_E(k\A_1)^*\otimes_EA\ar[r]^-{d_2^*}\ar[d]^-{\psi_2} & A\otimes_E (kA_2)^* \otimes_EA\ar[r]\ar[d]^-{\psi_3} & 0\\
  0\ar[r] & A\otimes_Ek\A_2\otimes_EA\ar[r]^-{\overline d_0} & A\otimes_Ek\A_1\otimes_EA\ar[r]^-{\overline d_1} & A\otimes_Ek\A_0\otimes_EA\ar[r]^-{\overline d_2} & A\otimes_EA\ar[r] & 0}
 \]}
 
 
 where $\overline d_0$ is given by 
 \begin{align*}
  \overline d_0(1\otimes d^2u^2\otimes 1) = 1\otimes du^2\otimes d- d\otimes du^2\otimes 1 - u\otimes d^2u\otimes 1+1\otimes d^2u\otimes u.
 \end{align*}
 
 $\overline d_1$ is
 \begin{align*}
  \overline d_1(1\otimes d^2u\otimes 1) &= 1\otimes d\otimes du - 
     \beta d\otimes d\otimes u - \beta d^2\otimes u\otimes 1\\
     &-\alpha(1\otimes d\otimes ud + d\otimes u\otimes d + du\otimes d\otimes 1)\\
     &-\beta(-\beta^{-1}\otimes u\otimes d^2 - \beta^{-1} u\otimes d\otimes d + ud\otimes d\otimes 1) - \gamma\otimes d\otimes 1 \\
  \overline d_1(1\otimes du^2\otimes 1)&= -\beta\otimes d\otimes u^2 - \beta d\otimes u\otimes u + du\otimes u\otimes 1\\
     &-\alpha(1\otimes u\otimes du + u\otimes d\otimes u + ud\otimes u\otimes 1)\\
     &-\beta(1\otimes u\otimes ud - \beta^{-1}u\otimes u\otimes d - \beta^{-1}u^2\otimes d\otimes 1) - \gamma\otimes u\otimes 1
 \end{align*}
 
 and $\overline d_2$ is
 \begin{align*}
  &\overline d_2(1\otimes u\otimes 1) = -\beta\otimes u - u\otimes 1,&
   &\overline d_2(1\otimes d\otimes 1) = 1\otimes d + \beta d\otimes 1,
 \end{align*}
 
 From this we deduce that $HH^3(A,A^e)\cong A\otimes_EA/(\im \overline d_2)$.
 Let $\sigma$ be the algebra automorphism of $A$ 
 defined by $\sigma(d)=-\beta d$, $\sigma(u)=-\beta^{-1}u$. Recall that $A_\sigma$ is the $A$-bimodule with $A$ as underlying vector space and action of $A\otimes_kA^{op}$ 
 given by: $(a\otimes b)\cdot x = ax\sigma(b)$,
 that is, it is twisted on the right by the automorphism $\sigma$.

 It is easy to see that if $\beta\neq 0$ then $ A_\sigma\cong A\otimes_EA/(\im \overline d_2)\cong HH^3(A,A^e)$ as $A$-bimodules. If $\beta=0$ then the action on the left by $u$ on $HH^3(A,A^e)$ is zero
 and then $A\ncong HH^3(A,A^e)$ since the action on the left by $u$ on $A$ 
is injective. We conclude after a short computation that 
$HH^3(A,A^e)\cong A$ if and only if $\beta=-1$.
 Notice that for $\beta=-1$ the complex in the second line of the diagram above is the resolution of $A$. As a consequence, $A$ is $3$-Calabi-Yau if and only if $\beta=-1$. In this case
  the potential $\Phi$ equals 
 $d^2u^2 + \frac{\alpha}{2}dudu + \gamma du$. For $\beta\neq 0,-1$, we shall 
 see in a forthcoming article that $A$ is twisted $3$-Calabi-Yau algebra 
\cite{BSW}, coming from a twisted potential.
\end{proof}

\section{Final remarks}
\label{s:Final remarks}

We have studied some examples of algebras, in particular of $N$-Koszul algebras for which we managed 
to obtain the minimal resolution using our methods. This fact can be 
stated in general as follows.
\begin{theorem}
\label{teo3}
 Given an algebra $A=kQ/I$ such that
 \begin{enumerate}
  \item there is a reduction system $\mathcal R=\{(s_i,f_i)\}_i$ for $I$ 
satisfying $(\lozenge)$ with $s_i$ and $f_i$ homogeneous of length 
$N\geq2$ for all $i$,
\label{finalremarksitem1}
  \item for all $n\in\NN$, the length of the elements of $\A_n$ is 
strictly smaller that the length of the elements of $\A_{n+1}$.
\label{finalremarksitem2}
 \end{enumerate}
The resolutions of $A$ as $A$-bimodule obtained using Theorem \ref{teo1} 
and Theorem \ref{teo2} are minimal.
\end{theorem}
\begin{proof}
 Let $(A\otimes_Ek\A_\bullet\otimes_EA,d_\bullet)$ be a resolution of 
$A$ as $A$-bimodule obtained using Theorem \ref{teo1} or Theorem 
\ref{teo2}. Denote by $|c|$ the length of a path $c\in Q_{\geq0}$. 
Condition \eqref{finalremarksitem1} guarantees that for all paths $p,q$ 
such that $\lambda p\preceq q$ for some $\lambda\in k^\times$, 
we have $|p| = |q|$. Let $n\geq0$,  $q\in\A_n$ and $\lambda \pi(b)\otimes 
p\otimes\pi(b')\in\overline{\mathcal L}^\prec_{n-1}(q)$. Since 
$p\in\A_{n-1}$, condition \eqref{finalremarksitem2} says that $|p|<|q|$. On 
the other hand, $\lambda bpb'\prec q$ and then $|bpb'| = |q|$. We deduce 
that $b\in Q_{\geq1}$ or $b'\in Q_{\geq1}$. As a consequence, $\im(d_n)$ is 
contained in the radical of $A\otimes_Ek\A_{n-1}\otimes_EA$ and therefore 
the resolution of $A$ is minimal.
\end{proof}
\begin{remark}
 The conclusion holds in a more general situation, which includes Example 
\ref{ej:qci}. It is sufficient to have a reduction system satisfying 
\eqref{finalremarksitem1} and such that the ambiguities $p$ that appear 
when reducing a given $n+1$-ambiguity $q$ are of length strictly smaller 
than the length of $q$.
\end{remark}
\begin{remark}
 In Example \ref{ej:cubica}, the reduction system $\mathcal R_2$ satisfies 
the conditions of Theorem \ref{teo3}, while $\mathcal R_1$ does not 
satisfy \eqref{finalremarksitem2}.
\end{remark}

Notice that if $\mathcal R$ is a reduction system for an algebra 
for which there is a non-resolvable ambiguity, then, even if we complete 
it like we did in Example \ref{ej:cubica}, the resolutions obtained using 
Theorem \ref{teo1} and Theorem \ref{teo2} will not be minimal.
 
\bigskip

We end this article proving a generalization of Prop. 8 of \cite{GM} and a corollary. 

\begin{proposition}
Let $A= kQ/I$, where $Q$ is a finite quiver, $kQ$ is the 
path algebra graded by the length of paths and $I$ a homogeneous ideal with respect to this grading, 
contained in $Q_{\ge 2}$. Let $\mathcal{R}$ be a reduction system satisfying conditions (1) and (2) of Theorem 
\ref{teo3}
and let $A_S$ be 
the associated monomial algebra. 
The algebra $A_S$ is $N$-Koszul if and only if $A$ is an $N$-Koszul algebra.
\end{proposition}
\begin{proof}
The projective bimodules appearing in the minimal resolution of $A_S$ are in one--to--one correspondence 
with those appearing in the resolution of $A$, so either both of them are generated in the correct 
degrees or none is.
\end{proof}

This proposition, together with Proposition \ref{cuadratico} and Thm. 3 of \cite{GH} give the following result.
\begin{corollary}
 If $A$ has a reduction system $\mathcal{R}$ 
 satisfying condition (1) of Theorem 
\ref{teo3}
 and such that $S \subseteq Q_{2}$, then $A$ is Koszul. 
\end{corollary}


\vspace{0.5 cm} 


Sergio Chouhy and Andrea Solotar: IMAS y Dto de Matem\'{a}tica, Facultad de Ciencias Exactas y Naturales,
Universidad de Buenos Aires, 
Ciudad Universitaria, Pabell\`{o}n 1,
(1428) Buenos Aires, Argentina

\emph{asolotar@dm.uba.ar}

\emph{schouhy@dm.uba.ar}

\end{document}